\documentclass[journal,twoside,web]{ieeecolor-arxiv}

\let\labelindent\relax

\usepackage{generic}
\usepackage[noadjust]{cite}
\usepackage{amsmath,amssymb,amsfonts,amsthm,pifont,tablefootnote,nicefrac,enumitem}
\usepackage{graphicx}
\usepackage[colorlinks=true,    
      urlcolor=black,     
      citecolor=blue,
      menucolor=black,    
      linkcolor=black,    
      breaklinks=true,
]{hyperref}
\usepackage{textcomp}
\usepackage{booktabs}
\def\BibTeX{{\rm B\kern-.05em{\sc i\kern-.025em b}\kern-.08em
    T\kern-.1667em\lower.7ex\hbox{E}\kern-.125emX}}
\DeclareMathOperator*{\argmin}{arg\,min}
\let\oldinf\inf
\renewcommand{\inf}{\oldinf\limits}
\let\oldmin\min
\renewcommand{\min}{\oldmin\limits}

\begin{document}

\newcommand{\Rl}[2]{\ensuremath{\mathbb{R}^{#1}_{#2}}}   
\newcommand{\R}{\ensuremath{\mathbb{R}}}
\newcommand{\Rlp}{\ensuremath{\mathbb{R}_{>0}}}
\newcommand{\Rlpc}{\ensuremath{\overline{\mathbb{R}}_{>0}}}
\newcommand{\Rlo}{\ensuremath{\mathbb{R}_{\geq 0}}}
\newcommand{\Rln}{\ensuremath{\mathbb{R}_{< 0}}}
\newcommand{\Zo}{\ensuremath{\mathbb{Z}_{\geq 0}}}
\newcommand{\Zp}{\ensuremath{\mathbb{Z}_{> 0}}}
\newcommand{\Np}{\ensuremath{\mathbb{N}_{> 0}}}
\newcommand{\N}{\ensuremath{\mathbb{N}}}                 
\newcommand{\Z}{\ensuremath{\mathbb{Z}}}
\newcommand{\id}{\ensuremath{\text{id}}}

\definecolor{bleucit}{rgb}{0.1,0.4,0.8}
\newcommand{\bleucit}{\textcolor{bleucit}}
\newcommand{\postoyanbleucit}{\textcolor{bleucit}}

\newcommand\cp[1]{{`\emph{#1}'}}
\newcommand{\blue}{\textcolor{blue}}

\newcommand{\cmark}{\ding{51}}%
\newcommand{\xmark}{\ding{55}}%

\newcommand{\dst}{\displaystyle}
\newcommand{\Linf}[1]{\ensuremath{\mathcal{L}^{#1}}}

\newcommand{\eg}{{\it e.g.}}

\newcommand{\Nesic}{Ne{\v{s}}i{\'c} }

\newcommand{\stopvi}{\text{stop}}
\newcommand{\final}{\text{final}}

\definecolor{blue_cv}{rgb}{0.09,0.35,0.78}

\newcommand{\KL}{\ensuremath{\mathcal{KL}}}
\newcommand{\K}{\ensuremath{\mathcal{K}}}
\newcommand{\Kinf}{\ensuremath{\mathcal{K}_{\infty}}}
\newcommand{\KK}{\ensuremath{\mathcal{KK}}}
\newcommand{\KN}{\ensuremath{\mathcal{KN}}}
\newcommand{\KKL}{\ensuremath{\mathcal{KKL}}}
\newcommand{\KLL}{\ensuremath{\mathcal{KLL}}}
\newcommand{\D}{\ensuremath{\mathcal{D}}}
\newcommand{\PD}{\ensuremath{\mathcal{PD}}}

\newcommand{\Cs}{\ensuremath{C_{\text{steady}}}}
\newcommand{\Ct}{\ensuremath{C_{\text{transient}}}}
\newcommand{\Ds}{\ensuremath{D_{\text{steady}}}}
\newcommand{\Dt}{\ensuremath{D_{\text{transient}}}}
\newcommand{\UtGpAS}{U$_{\text{t}}$GpAS}
\newcommand{\UtGAS}{U$_{\text{t}}$GAS}
\newcommand{\UjGpAS}{U$_{\text{j}}$GpAS}
\newcommand{\UjGAS}{U$_{\text{j}}$GAS}

\newenvironment{romanlist}
  {\begin{enumerate}[label=(\roman*)]}
  {\end{enumerate}}

\newcommand{\arginf}{\ensuremath{\text{arginf}\,}}
\newcommand{\interior}{\ensuremath{\text{int}\,}}
\newcommand{\dom}{\ensuremath{\text{dom}\,}}
\newcommand{\Span}{\ensuremath{\text{Span}}}
\newcommand{\avg}{\ensuremath{\text{avg}}}
\newcommand{\co}{\ensuremath{\text{co}\,}}
\newcommand{\coc}{\ensuremath{\overline{\text{dom}}\,}}
\newcommand{\ext}{\ensuremath{\text{ext}}}
\newcommand{\rge}{\ensuremath{\text{rge}\,}}
\newcommand{\esup}{\ensuremath{\text{ess.sup}\,}}

\newcommand{\sign}[1]{\ensuremath{\text{sign}{(#1)}}}
\newcommand{\sat}{\ensuremath{\text{sat}}}

\newcommand{\sinc}{\ensuremath{\text{sinc}}}
\newcommand{\nom}{\ensuremath{\text{nom}}}

\newcommand{\Tmati}{\ensuremath{T_{MATI}\,\,}}
\newcommand{\Tmasp}{\ensuremath{\mathrm{T_{MASP}}\,}}
\newcommand{\lc}{\ensuremath{\llbracket}}
\newcommand{\rc}{\ensuremath{\rrbracket}}

\newcommand{\rx}{\ensuremath{r_{\mathsf{x}}}}
\newcommand{\barrx}{\ensuremath{\bar{r}_{\mathsf{x}}}}
\newcommand{\rw}{\ensuremath{r_{\mathsf{w}}}}
\newcommand{\rV}{\ensuremath{r_{\mathsf{V}}}}

\newcommand{\norm}[1]{\ensuremath{\left\|{#1}\right\|}}
\newcommand{\ip}[2]{\ensuremath{\left\langle #1, #2\right\rangle}}
\newcommand{\cb}[1]{\ensuremath{\overline{\mathbb{B}}_{\mbox{\scriptsize $#1$}}}}                              
\newcommand{\ob}[2]{\ensuremath{\mathbb{B}_{\mbox{\scriptsize $#1$}}\ensuremath{\left( #2\right)}}}   
\newcommand{\df}{\ensuremath{\stackrel{\mbox{\tiny $\mathrm{def}$}}{=}\:}}                                             
\newcommand{\myint}[4]{\ensuremath{\int_{#1}^{#2}#3\;\mathrm{d}#4}}
\newcommand{\Mm}{\ensuremath{\:\stackrel{\rightarrow}{\scriptstyle{\rightarrow}}\:}}
\newcommand{\bm}[1]{\ensuremath{\mathbf{#1}}}
\newcommand{\hs}[1]{\hspace*{#1 em}}
\newcommand{\qa}{\ensuremath{\mathcal{Q}_{A}}}%
\newcommand{\mc}[1]{\ensuremath{\mathcal{#1}}}
\newcommand{\di}{\ensuremath{\mathcal{D}_{i}}}

\newcommand{\HS}{\ensuremath{\mathcal{H}}}
\newcommand{\HSc}{\ensuremath{\mathcal{H}_c}}

\newcommand{\Sx}{\ensuremath{\mathcal{S}_{\mathcal{X}}}}


%

\newcommand{\ie}{{\it i.e. }}

\newtheorem{exple}{Example}
\newtheorem{defn}{Definition}
\newtheorem{claim}{Claim}
\newtheorem{hypo}{Hypothesis}
\newtheorem{ass}{Assumption}
\newtheorem{fact}{Fact}
\newtheorem{lem}{Lemma}
\newtheorem{ex}{Example}
\newtheorem{thm}{Theorem}
\newtheorem{prop}{Proposition}
\newtheorem{cond}{Condition}
\newtheorem{cor}{Corollary}
\newtheorem{pb}{Problem}
\newtheorem{rem}{Remark}
\newtheorem{sass}{Standing Assumption}


%
%
\newenvironment{rems}{\textit{Remarks. }}{\mbox{}\\[1ex]}

\newenvironment{eqn}[1][]{%
    \ifx&#1&\else\label{#1}\fi%
    \equation
    \setlength{\arraycolsep}{1.5pt} 
    \begin{array}{rlllll}%
}{%
    \end{array}\endequation%
}

\title{Value iteration with stopping criterion:  finite  iterations, stability, and near-optimality guarantees}

\author{M. Granzotto,  R. Postoyan, D. Ne{\v{s}}i{\'c}, L. Buşoniu,  J. Daafouz 
\thanks{Work funded by the ANR under grant OLYMPIA ANR-23-CE48-0006,
the ARC under the Discovery Project DP250100300,
and the Romanian Hub for Artificial Intelligence-HRIA, Smart Growth, Digitization and Financial Instruments Program, MySMIS no. 351416.}
\thanks{M. Granzotto and D. Ne{\v{s}}i{\'c} are with the Department of Electrical and Electronic Engineering, University of Melbourne, Parkville, VIC 3010, Australia (e-mail: \{mgranzotto, dnesic\}@unimelb.edu.au).}
\thanks{R. Postoyan and J. Daafouz are with the Université de Lorraine,
CNRS, CRAN, F-54000 Nancy, France (e-mails: {name.surname}@univ-lorraine.fr). J. Daafouz is also with IUF.}
\thanks{L. Buşoniu is with the Technical University of Cluj-Napoca, 400114
Cluj-Napoca, Romania (e-mail: lucian.busoniu@aut.utcluj.ro) and is a Corresponding Member of the Romanian Academy.}}

\maketitle

\begin{abstract}
Value iteration (VI) is a cornerstone of dynamic programming that allows computing near-optimal feedback laws for general plant dynamics and cost functions. In practice, however, it must be stopped after finitely many iterations. This raises the question of when to stop the algorithm so that the resulting policies and value functions achieve desirable properties, like given near-optimality bounds and stability. 
In this context, we study deterministic, discrete-time systems with infinite-horizon (possibly discounted) costs whose inputs are generated by VI. We equip VI with a generalized stopping criterion that encompasses existing choices while allowing new ones. Our aim is to analyze the properties of the  policies and value functions at the final iteration. Under mild assumptions, we first show that VI indeed terminates in a finite number of iterations. We then establish that the final policies are stabilizing by properly designing the stopping criterion, and derive explicit near-optimality bounds characterized by this choice. These results offer a design framework for the  stopping criteria that balances computational effort with stability and performance guarantees.
\end{abstract}

\maketitle

\section{Introduction}\label{sect:introduction}

Value Iteration (VI) is one of the pillars of dynamic programming \cite{bertsekas2012dynamic}, with applications across numerous fields, including control engineering, reinforcement learning, and operations research, see, e.g., \cite{SuttonBarto:98,bertsekas2019reinforcement,lewis2012csm,puterman2014markov}. VI approximately solves optimal control problems for a broad class of dynamical systems and cost functions. As its name suggests,  VI consists in iterating on value functions with the goal of converging towards the optimal value function. The constructed value functions may then be used to obtain near-optimal policies, i.e., feedback laws. In practice, we cannot iterate infinitely many times and often VI is terminated using a so-called stopping criterion. This raises the question of how to design the stopping criterion such that: (i) it is satisfied after a finite number of iterations; and (ii) the resulting policies and value function approximations satisfy desired properties. A customary approach in the dynamic programming literature consists in monitoring whether the difference between the two successive value functions is less than some strictly positive constant in terms of the infinity norm\footnote{Other stopping criteria have been considered in related but different contexts, such as policy iteration (PI) for linear-quadratic systems \cite{kleinman1968iterative, bian-et-al-aut14, lamperski2020computing}, or nonlinear model predictive control (MPC) utilizing interior-point solvers \cite{pavlov2019early}. Similar open questions regarding termination and closed-loop guarantees as mentioned afterwards largely apply to these frameworks as well.}, in which case VI is stopped  (see, e.g., \cite[Chap. 4.4]{SuttonBarto:98}, \cite[p.161]{puterman2014markov}, \cite{al-tamimi-tsmc08, heydari2014revisiting, wei-liu-tc2015, kvretinsky2023stopping}). Alternative stopping criteria have been proposed for specific classes of systems and stage costs, like the linear-quadratic problem \cite{bian-jiang-aut2016(vi)}. However, the existence of a finite iteration at which the stopping criterion is satisfied  
is  often eluded, despite its importance. This property is not as trivial as it may seem for general systems and stage costs: it requires the sequence of value functions given by VI  to converge uniformly to the optimal one, while this convergence is only shown to be point-wise in general \cite{bertsekas-tnnls15}.  Moreover and importantly, when the stopping criterion is satisfied, it remains an open question what properties are exhibited by the resulting policy and value function approximation. 

In this context, we study deterministic discrete-time systems and infinite-horizon cost functions that may be discounted. Our goal is to provide conditions under which VI equipped with a stopping criterion: (i) is guaranteed to terminate in a finite number of iterations, (ii)  ensures the  policies obtained at the last iteration are stabilizing; and (iii) is such that the mismatch between the final value function and the optimal value function can be tuned by adjusting the stopping criterion. For this purpose, we make well-posedness  assumptions on the system model, the stage cost and the set of admissible inputs. We also impose detectability and stabilizability assumptions on the plant relative to the stage cost 
in line with, e.g.,\cite{Grimm-et-al-tac2005,Postoyan-et-al-tac(optimal),Granzotto-et-al-tac-finite-discounted-horizon,grune2012nmpc,Gaitsgory-et-al-16}, and consistent with the standard assumptions adopted in LQ control \cite{anderson-moore-book-optimal}.
We emphasize that all these assumptions can be tested  before running the algorithm. We then establish that VI is recursively feasible in the sense that the optimization problem solved at each iteration  always admits a solution. 

Afterwards, we formalize the considered class of stopping criteria, which consists in checking whether, for all states in a given set of interest $\mathcal{X}$, the absolute difference between two successive value functions is less than a prescribed bound. This bound is defined by a function $c_\stopvi$, which may depend on the state, the discount factor (if any), the number of iterations, and some tunable parameters.  In this way, we cover the commonly used criteria in the dynamic programming literature and enable the design of new ones that may reduce the required number of iterations while ensuring stability, as illustrated by a numerical example.  We provide two sets of conditions depending on the properties they induce, namely either a finite number of iterations or stability;  near-optimality guarantees can be derived for both of these. While these objectives are often pursued separately, we demonstrate that they are  compatible: we can always design stopping criteria that simultaneously satisfy  both requirements. 
The contributions are as follows. 

First, we establish that VI is guaranteed to terminate in a finite number of iterations thereby filling a gap in the literature. 
This is shown under a broad class of stopping criteria, encompassing those commonly used in practice for VI and also for PI \cite{kleinman1968iterative, bian-et-al-aut14, lamperski2020computing}.

Second, we analyse the stability properties of the plant constrained to the set $\mathcal{X}$, where the stopping criterion is evaluated, and whose inputs are generated by a final policy.  We show that these closed-loop systems exhibit a practical stability property where the adjustable parameters are the parameters of $c_\stopvi$. Stronger stability properties, namely uniform  global asymptotic and uniform global exponential stability properties, are established under extra conditions on the stopping criteria and some of the functions used to formulate the stabilizability and detectability properties. As for the first contribution, the question of the stabilizing properties of the policies at the last iteration is often eluded in the literature. We provide here an analysis that unravels the impact of the choice of the stopping criterion on  closed loop 
stability. Furthermore, this analysis applies to a broad class of stopping criteria, rather than a specific one.  These results also generalize our previous work in \cite{Granzotto-et-al-tac-finite-discounted-horizon} to non-zero initial value functions.

Third, we derive near-optimality bounds involving the stopping criterion. Contrary to the near-optimality bounds found in the dynamic programming literature \cite{bertsekas2012dynamic}, the provided bounds do not explode when the discount factor is equal to $1$, like in \cite{Granzotto-et-al-tac-finite-discounted-horizon}. Compared to \cite{Granzotto-et-al-tac-finite-discounted-horizon}, these bounds depend on the stopping criterion, which can thus be adjusted so that the final value function  lies within a given range from the optimal value function. Notably, the majority of our results do not require the initial value function to be either smaller than or equal to, or larger than or equal to, the optimal value function, as is commonly assumed in the literature  \cite[Prop. 2]{bertsekas-tnnls15}. 

A limitation of our work is that VI is assumed to be computed exactly at each iteration. This allows us to focus on the main challenge of designing stopping criteria ensuring finite iteration counts, stabilizing policies, and near-optimality guarantees. Extensions to account for computational errors in VI are left to future work; however, our numerical results suggest that the findings remain relevant in their presence.

Compared to the preliminary version \cite{granzotto-et-al-stop-l4dc}, the main novelties are: (i) discounted cost functions; (ii) non-zero initial value functions, which may accelerate the convergence of VI; and (iii) more general stopping criteria depending on the discount factor and the number of iterations. These extensions introduce major technical challenges and broaden the scope of the results.


The remainder of this work is organized as follows. Section \ref{sect:problem-statement} formalizes the problem, while Section \ref{sect:assumptions} presents the assumptions. The stopping criteria are introduced in Section \ref{sect:stopping-criteria}. Section \ref{sect:finite-iterations} addresses the existence of a finite number of iterations, Section \ref{sect:stability} studies the stability of the closed-loop system at the final iteration, and Section \ref{sect:near-optimality} provides the near-optimality analysis. A numerical example is presented in Section \ref{sect:example}, and Section \ref{sect:conclusion} concludes the paper. Technical results and lengthy proofs are deferred to the appendix.\\

\noindent\textbf{Notation and preliminaries.} The symbol $\R$ stands for the set of real numbers, and $\Rlo$ ($\Rlp$) for the set of non-negative (positive) real numbers.  The symbol $\Zo$ ($\Zp$) denotes the set of non-negative (positive) integers. We use $\emptyset$ to denote the empty set. The identity map from a set $\mathcal{S}$ to itself is denoted $\id$. The ceiling function defined from $\R$ to $\Z$ is denoted $\lceil \cdot\rceil$. We use $|\cdot|$ to denote the Euclidean norm. Given a map $f:\R^{n}\to\R^m$ with $n,m\in\Zp$ and a set $\mathcal{S}\subset\R^{m}$, $f^{-1}(\mathcal{S})=\{x\in\R^{n}\,:\,f(x)\in \mathcal{S}\}$. Given a map $f:\R^{n}\to\R^n$ with $n\in\Zp$ and $k\in\Zo$, $f^{(k)}$ stands for the $k^{\text{th}}$ composition of $f$ with itself  
with $f^{(0)}=\id$. Given two sets $\mathcal{S},\mathcal{R}$ and a set-valued map $F:\mathcal{S}\rightrightarrows \mathcal{R}$, we say that the single-valued map $f:\mathcal{S}\to\mathcal{R}$ is a selection of $F$ if $f(x)\in F(x)$ for any $x\in\mathcal{S}$.  The range of set-valued map $F:\mathcal{S}\rightrightarrows \mathcal{R}$ is denoted $\rge F=\{y\in\mathcal{R}\,:\, \exists x\in\mathcal{S},\,y\in F(x)\}$. Given a possibly infinite-length sequence $\mathbf{z}=(z_0,z_1,\ldots)\in(\R^{n_z})^{\mathcal{Z}}$ with $\mathcal{Z}=\{1,\ldots,N\}$ with $n_z\in\Zp$ for some $N\in\Zp$ or $\mathcal{Z}=\Zp$, $\mathbf{z}\vert_k$ stands for the truncation of $\mathbf{z}$ to its first $k\in\Zo$ elements, i.e., $\mathbf{z}\vert_k=(z_0,z_1,\ldots,z_{k-1})$ with $\mathbf{z}\vert_0=\emptyset$.  Given any two vectors $x\in\R^{n_x}$ and $y\in\R^{n_y}$ with $n_x,n_y\in\Zp$, $(x,y)$ stands for $(x^\top,y^\top)^\top$. We consider sets $\K$, $\Kinf$ and $\KL$ for comparison functions as defined in \cite[Chap. 3]{Goebel-Sanfelice-Teel-book}. 
Inspired by \cite{sanfelice-teel-aut10}, a subset $E\subset\Zo$ is called a \emph{discrete time domain} if $E=\{0,\ldots,k\}$ for some $k\in\Zo$ and $\phi:\dom\phi\to\R^n$ with $n\in\Zp$ is a \emph{discrete arc} if $\dom \phi$ is a discrete time domain. We will sometimes write ``s.t.'' in place of ``such that'' for space reasons.

\section{Problem statement}\label{sect:problem-statement}

We first present the considered class of plant models and cost functions and recall VI (Section \ref{subsect:system,cost}). Afterwards, we state the main objectives of this work (Section \ref{subsect:objectives}).

\subsection{System, cost function and value iteration}\label{subsect:system,cost}

We consider deterministic discrete-time systems of the form
\begin{eqn}\label{eq:sys}
x(k+1) & = & f(x(k),u(k)),    
\end{eqn}
where $x(k)\in\R^{n_x}$ is the state, $u(k)\in\mathcal{U}(x(k))\subset\R^{n_u}$ is the control input at time $k\in\Zo$ and $n_x,n_u\in\Zp$. Set-valued map $\mathcal{U}:\R^{n_x}\rightrightarrows\R^{n_u}$ defines the set of \emph{admissible} inputs at a given state.  Hence, we say that an input $u$ is admissible at a given state $x$ if $u\in\mathcal{U}(x)$. For the sake of convenience, we introduce the set $\mathcal{W}:=\{(x,u)\,:\,u\in\mathcal{U}(x)\}\subset\R^{n_x\times n_u}$. We denote the solution to (\ref{eq:sys}) at time $k\in\Zo$ initialized at $x\in\R^{n_x}$ at time $0$ with the infinite-length admissible\footnote{In the sense that $u_k\in\mathcal{U}(\phi(k;x,(u_0,\ldots,u_{k-1})))$ for any $k\in\Zo$.} control input sequence $\mathbf{u}\in(\R^{n_u})^{\Zp}$  as $\phi(k;x,\mathbf{u}\vert_k)$ with the convention $\phi(0;x,\mathbf{u}\vert_0)=x$, where $\phi$ is a discrete arc defined on $\dom\phi(\cdot;x,\mathbf{u})=\Zo$. Throughout this study, we will consider system (\ref{eq:sys}) for which the inputs are generated by different (set-valued) feedback laws. We will thus introduce distinct {solution} notations  accordingly that are summarized in Table \ref{tab:solutions}.

We focus on the scenario where the control inputs to (\ref{eq:sys}) are given by policies generated by VI. We consider for this purpose the next infinite-horizon cost function defined, for initial state $x\in\R^{n_x}$ and infinite-length sequence of admissible control inputs $\mathbf{u}\in\left(\R^{n_u}\right)^{\Zp}$, as
\begin{eqn}\label{eq:cost}
J_{\gamma}(x,\mathbf{u}) & = & \dst\sum_{k=0}^{\infty}\gamma^{k}\ell(\phi(k;x,\mathbf{u}\vert_k),u_k)
\end{eqn}
where $\ell:\R^{n_x}\times\R^{n_u}\to\Rlo$ is the stage cost and $\gamma\in(0,1]$ is the discount factor with some slight abuse of terminology as we allow $\gamma=1$. We define the optimal value function associated with cost (\ref{eq:cost}) as
\begin{eqn}\label{eq:optimal-value-function}
V_{\gamma,\star}(x) & = &\inf_{\mathbf{u}} \dst\sum_{k=0}^{\infty}\gamma^{k}\ell(\phi(k;x,\mathbf{u}\vert_k),u_k) & & \forall x\in\R^{n_x},
\end{eqn}
{where the elements of the infinite-length sequence $\mathbf{u}$ take value in the  admissible set of inputs at the corresponding state, i.e., $u_k \in \mathcal{U}({\phi}(k;x,(u_0,\ldots,u_{k-1})))$ for any $k\in\Zo$.}

To determine $V_{\gamma,\star}$ and the associated optimal policies (when these exist) is notoriously hard in general. VI instead consists in iteratively constructing increasingly better approximations of the optimal value function that converge towards $V_{\gamma,\star}$. In particular, given an initial value function $V_{\gamma,0}:\R^{n_x}\to\Rlo$,  VI consists  in updating the value function at each iteration $i\in\Zp$ as, for any $x\in\R^{n_x}$,
\begin{eqn}\label{eq:vi-value}
V_{\gamma,i}(x) & := & \inf_{u\in\mathcal{U}(x)}\left(\ell(x,u)+\gamma V_{\gamma,i-1}(f(x,u))\right).
\end{eqn}
When the infimum above is a minimum, under conditions  established later, we retrieve policies from the obtained value function by taking any selection of the set-valued map 
\begin{eqn}\label{eq:vi-policy}
H_{\gamma,i}  : \R^{n_x} & \rightrightarrows & \R^{n_u} \\
x & \mapsto & \argmin\limits_{u\in\mathcal{U}(x)}\left(\ell(x,u)+\gamma V_{\gamma,i-1}(f(x,u))\right).
\end{eqn}
The above map is set-valued as, in general, there is no guarantee that the considered argmin is unique; this fact is captured by the forthcoming analysis.


\subsection{Objectives}\label{subsect:objectives}

Under mild conditions recalled in the sequel, the sequence of functions $V_{\gamma,i}$ given by (\ref{eq:vi-value}) converges (point-wisely) towards $V_{\gamma,\star}$ as defined in (\ref{eq:optimal-value-function}), like in \cite{bertsekas-tnnls15}. This property is only asymptotic and, when implementing VI, we do not iterate infinitely many times. As a result, the question of when to stop iterating VI arises. This question is intimately related to the desired properties we want the policies generated by VI to ensure for system (\ref{eq:sys}) at the final iteration. In this work, our goal is to first ensure that the number of iterations needed to satisfy the stopping criterion  is finite. We then aim at ensuring that  (i) the final policy is stabilizing for system (\ref{eq:sys}) in a sense made precise later, and (ii)  the mismatch between the final value function and the optimal value function (\ref{eq:optimal-value-function}) is less than a given bound involving the stopping criterion. By adjusting the stopping criterion, we can  tune this bound to establish explicit near-optimality guarantees. We use certain assumptions,  which are presented in the next section.

\section{Assumptions}\label{sect:assumptions}

It is first of all essential to ensure  that a policy can be computed at any iteration of VI. In other words, we need  to be sure that, at any iteration and at any state, the set-valued map $H_{\gamma,i}$ in (\ref{eq:vi-policy}), used to construct policies, is non-empty; see \cite{granzotto2024-tac(pi),bertsekas-tnnls15} for related results. This property is sometimes referred to as  (recursive) feasibility  by adopting the terminology used in the model predictive control literature \cite{mayne2000survey}, as done in \cite{granzotto2024-tac(pi)} for PI. We make for this purpose what we call well-posedness assumptions on the problem (Section \ref{subsect:regularity}). Then, we make a series of additional assumptions to  establish the desired stability and near-optimality properties (Section \ref{subsect:detectability-stabilizability}). As we will see, all these assumptions are  reasonable in the context of this work and can be tested given the initial data of the problem, namely $f$ and $\mathcal{U}$ in (\ref{eq:sys}), $\ell$ in (\ref{eq:cost}) and the initial value function $V_{\gamma,0}$ in (\ref{eq:vi-value}). We conclude this section by deriving models of system (\ref{eq:sys}) in closed-loop either with inputs generated by policies obtained by VI or with optimal inputs  (Section \ref{subsect:closed-loop}), as both systems play a key role. 

As a preliminary step, we introduce the continuous, surjective function $\sigma:\R^{n_x}\to\Rlo$. This function $\sigma$ is used to define stability as in e.g., \cite{Grimm-et-al-tac2005,Postoyan-et-al-tac(optimal),Granzotto-et-al-tac-finite-discounted-horizon}. When studying the stability of the origin, we can take $\sigma$ such that it is positive definite and radially unbounded, like $\sigma(x)=|x|^a$ for any $x\in\R^{n_x}$ for a given $a>0$, or $\sigma(x)=x^\top P x$ for any $x\in\R^{n_x}$ with $P$ real, symmetric, positive definite. On the other hand, when studying the stability of a compact set $\mathcal{A}\subset\R^{n_x}$, $\sigma$ has to vanish in $\mathcal{A}$, to be strictly positive elsewhere and radially unbounded, like when $\sigma(x)=\min\{|x-z|\,:\,z\in\mathcal{A}\}$. In all cases, $\sigma^{-1}(0)$ defines the attractor of interest and the function $\sigma$ allows covering a range of attractors in a unified way. 

\subsection{Well-posedness}\label{subsect:regularity}

We make the next assumptions on $f$, $\mathcal{U}$, $\ell$ and $V_{\gamma,0}$.

\begin{sass}[\textbf{SA\ref{sass:well-posedness}}]\label{sass:well-posedness} The following holds.
\begin{enumerate}[label=(\roman*)]
\item $f$ is continuous on $\R^{n_x}\times\R^{n_u}$.
\item {For any $x\in\R^{n_x}$, $\mathcal{U}(x)$ is non-empty and the set $\mathcal{W}=\{(x,u)\in \R^{n_x}\times\R^{n_u}\,:\,u\in\mathcal{U}(x)\}$ is closed.}
\item $\ell$ is lower semicontinuous\footnote{See \cite[Definition 1.5]{Rockafellar-Wets-book}.} on $\R^{n_x}\times\R^{n_u}$ and there exists  $\alpha_\ell\in\Kinf$ such that $\ell(x,u)\geq \alpha_\ell(|u|)$ for any $(x,u)\in\R^{n_x\times n_u}$.
\item $V_{\gamma,0}$ is lower semicontinuous on $\R^{n_x}$.
\end{enumerate}
\end{sass}

Most of the properties of SA\ref{sass:well-posedness} are related to the regularity of the problem. We note that the second part of  SA\ref{sass:well-posedness}(iii) holds whenever $\ell(x,u)=\ell_1(x,u)+\ell_2(u)$ with $\ell_1,\ell_2$ taking non-negative values and $\ell_2$ verifying $\ell_2(u)\geq \alpha_\ell(|u|)$ for any $u\in\R^{n_u}$ with  $\alpha_\ell\in\Kinf$ for instance; an  example being $\ell_2(u)=u^\top R u$ with $R$  a real, {symmetric}, positive definite matrix.

A  consequence of SA\ref{sass:well-posedness} is that  the infimum in (\ref{eq:vi-value}) is a minimum at every iteration of VI. This property is essential to allow synthesizing policies using (\ref{eq:vi-policy}) at each iteration as already mentioned.

\begin{thm}\label{th:recursive-feasibility} Given any iteration $i\in\Zp$ and discount factor $\gamma\in(0,1]$, for any $x\in\R^{n_x}$, $V_{\gamma,i}(x)=\min_{u\in\mathcal{U}(x)}\left(\ell(x,u)+{\gamma}V_{\gamma,i-1}(f(x,u))\right)$ and $H_{\gamma,i}(x)\neq \emptyset$.
\end{thm}

\begin{proof} The fact that $H_{\gamma,i}(x)\neq\emptyset$ for any $i\in\Zp$, $\gamma\in(0,1]$ and $x\in\R^{n_x}$ follows from the fact that $V_{\gamma,i}(x)=\min_{u\in\mathcal{U}(x)}\big(\ell(x,u)+\gamma V_{\gamma,i-1}(f(x,u))\big)$, we therefore only have to prove this latter property. We will invoke for this purpose \cite[Thm. 1]{Keerthi-Gilbert-tac85}.  
Let $i\in\Zp$ and $x\in\R^{n_x}$. 
We observe that
\begin{equation}
\begin{aligned}
    V_{\gamma,i}(x) &= \inf_{u\in\mathcal{U}(x)} \big( \ell(x,u) + \gamma V_{\gamma,i-1}(f(x,u)) \big) \\
    &= \inf_{u\in\mathcal{U}(x)} \big( \ell(x,u) + \gamma\inf_{u'\in\mathcal{U}(f(x,u))} \big( \ell(f(x,u),u') \\
    &\quad\quad + \gamma V_{\gamma,i-2}\bigl(f(f(x,u),u')\bigr) \big) \big) \\
    &= \inf_{u\in\mathcal{U}(x)} \big( \ell(x,u) + \gamma \inf_{u'\in\mathcal{U}(\phi(1;x,u))} \big( \ell(\phi(1;x,u),u') \\
    &\quad\quad + \gamma V_{\gamma,i-2}\bigl(\phi(2;x,(u,u'))\bigr) \big) \big).
\end{aligned}
\end{equation}
Hence
\begin{equation}\label{eq:proof-vi-as-mpc}
\begin{array}{rclll}
V_{\gamma,i}(x) & = &   \inf_{\mathbf{u}}\big(\sum_{k=0}^{i-1}\gamma^k \ell(\phi(k;x,\mathbf{u}\vert_k),u_k)\\
& & \quad\quad+\gamma^{i}V_{\gamma,0}(\phi(i;x,\mathbf{u}\vert_i))\big),
\end{array}
\end{equation}
{where the infimum is taken over the admissible input sequences $\mathbf u=(u_0,\dots,u_{i-1})$, i.e., $u_k\in\mathcal U(\phi(k;x,\mathbf u|_k))$ for $k\in\{0,\dots,i-1\}$.}


We now check that the conditions of \cite[Thm. 1]{Keerthi-Gilbert-tac85} are satisfied for system (\ref{eq:sys}) and cost function (\ref{eq:proof-vi-as-mpc}). {We take for this purpose, with the notation of \cite{Keerthi-Gilbert-tac85}, $X=\R^{n_x}$, $U=\R^{n_u}$, $T_0=\{x\}$, $f_k=f$ for any $k\in\Zo$, $W_k=\mathcal{W}$ and $g_k=\gamma^k\ell$ for $k\in\{0,\ldots,i-1\}$, $W_k=\R^{n_x}\times\{0\}$ for $k\geq i$, $g_i(z,u)=\gamma^iV_{\gamma,0}(z)$ for $(z,u)\in W_i$, and $g_k\equiv0$ for $k>i$}. Item a) of \cite[Thm. 1]{Keerthi-Gilbert-tac85} holds by SA\ref{sass:well-posedness}(ii). Item b) of \cite[Thm. 1]{Keerthi-Gilbert-tac85} corresponds to SA\ref{sass:well-posedness}(i). Item c) of \cite[Thm. 1]{Keerthi-Gilbert-tac85} holds by  SA\ref{sass:well-posedness}(iii)-(iv). 
{Item d) of \cite[Thm. 1]{Keerthi-Gilbert-tac85} is replaced by item d$_3$) of \cite[Thm. 2]{Keerthi-Gilbert-tac85} for $k\in\{0,\ldots,i-1\}$, with $\phi_k=\gamma^k\alpha_\ell(|\cdot|)$ by SA\ref{sass:well-posedness}(iii), and by item d$_2$) of \cite[Thm. 2]{Keerthi-Gilbert-tac85} with $\Omega_k=\{0\}$ for $k\geq i$.} 
Finally, {as the considered cost is finite-horizon and the set-valued map $\mathcal{U}$ is non-empty on $\R^{n_x}$ by SA\ref{sass:well-posedness}(ii), {there exist admissible input sequences for any $x\in\R^{n_x}$, and the cost in (\ref{eq:proof-vi-as-mpc}) of
any such sequence is finite as $\ell$ and $V_{\gamma,0}$ take finite values}: item e) of \cite[Thm. 1]{Keerthi-Gilbert-tac85} holds}. We therefore apply \cite[Thm. 1]{Keerthi-Gilbert-tac85} to conclude that the infimum can be replaced by a minimum in (\ref{eq:proof-vi-as-mpc}), i.e., 
\begin{eqn}\label{eq:vi-as-finite-horizon}
V_{\gamma,i}(x) & = & \min_{\mathbf{u}}\big(\sum_{k=0}^{i-1}\gamma^{k}\ell(\phi(k;x,\mathbf{u}\vert_k),u_k) \\
& & \quad\quad\quad\quad+\gamma^i V_{\gamma,0}(\phi(i;x,\mathbf{u}\vert_i))\big),
\end{eqn}
{where the minimum is taken over admissible input sequences as in (\ref{eq:proof-vi-as-mpc}).} The desired result, namely that $V_{\gamma,i}(x)=\min_{u\in\mathcal{U}(x)}(\ell(x,u)+\gamma V_{\gamma,i-1}(f(x,u)))$ for any state $x\in\R^{n_x}$ and iteration $i\in\Zp$, follows from (\ref{eq:vi-as-finite-horizon}). 
\end{proof} 

\begin{rem}\label{rem:vi-mpc}
Equation (\ref{eq:proof-vi-as-mpc}) recalls a well-known fact of the literature, namely that VI solves at iteration $i$ a finite-horizon optimal control problem of length $i$ with terminal cost $V_{\gamma,0}$ similarly to a MPC problem, see e.g., \cite{bertsekas2013rollout,bertsekas2024model,Granzotto-et-al-tac-finite-discounted-horizon}.
\end{rem}

\subsection{Cost detectability and stabilizability}\label{subsect:detectability-stabilizability}

We make the next assumption on the stage cost $\ell$ in  (\ref{eq:cost}), which is related to a  detectability property of system (\ref{eq:sys}) with respect to $\ell$ as explained below.

\begin{sass}[\textbf{SA\ref{sass:detectability}}]\label{sass:detectability} There exists $\underline\alpha\in\Kinf$ such that for any $(x,u)\in\mathcal{W}$, $\ell(x,u)\geq \underline\alpha(\sigma(x))$.
\end{sass}

SA\ref{sass:detectability} holds in a common case where $\ell(x,u)=\ell_1(x)+\ell_2(x,u)$ with $\ell_1(x)\geq \underline\alpha(\sigma(x))$ for any $x\in\R^{n_x}$ and $\ell_2$ taking non-negative values. {A typical example  is $\ell_1(x)=x^\top Q x$ {and $\ell_2(u)=u^\top R u$} with $Q,R$ real, symmetric, positive definite matrices, which also satisfies SA\ref{sass:well-posedness}(iii).}  Given SA\ref{sass:detectability}, we have that when $\ell(x,u)=0$, necessarily $\sigma(x)=0$, which means the state $x$ is in the attractor set $\sigma^{-1}(0)$. Stage cost $\ell$ thus allows to detect whether the state $x$ lies in the attractor $\sigma^{-1}(0)$, as well as when it is close to it, as a small value of $\ell$ implies a small value of $\sigma(x)$,   since $\underline\alpha\in\Kinf$ and $\sigma$ is continuous. 

We also make the next assumption on $V_{\gamma,i}$ at any iteration $i\in\Zo$ and on the optimal value function $V_{\gamma,\star}$.

\begin{sass}[\textbf{SA\ref{sass:stabilizability}}]\label{sass:stabilizability} There exists $\overline\alpha\in\Kinf$ such that for any iteration $i\in\Zo$, any discount factor $\gamma\in(0,1]$ and any $x\in\R^{n_x}$, $V_{\gamma,i}(x)\leq\overline\alpha(\sigma(x))$ and $V_{\gamma,\star}(x)\leq \overline\alpha(\sigma(x))$.
\end{sass}

SA\ref{sass:stabilizability} is related to the stabilizability property of system (\ref{eq:sys}) with respect to the costs (\ref{eq:cost}) and (\ref{eq:vi-value}). This relationship is justified in \cite[Section III]{Grimm-et-al-tac2005} via the connection between VI and MPC recalled in Remark \ref{rem:vi-mpc}. 
As such, we emphasize that we do not need to compute $V_{\gamma,i}$ for $i\in\Zp$ or $V_{\gamma,\star}$ to determine whether SA\ref{sass:stabilizability} holds. Indeed, the knowledge of   sequences of admissible inputs making $\ell$ exponentially decreasing along the solutions to system (\ref{eq:sys}) is enough, as formalized next. 

\begin{lem}\label{lem:sufficient-conditions-for-sass-stabilizability}
Suppose that there exist $\nu\in(0,1)$, $c\geq 1$ and $\alpha_s\in\Kinf$ such that the following holds.
\begin{enumerate}[label=(\roman*)]
\item For any $\gamma\in(0,1]$, any $x\in\R^{n_x}$ and any $u\in\mathcal{U}(x)$, $V_{\gamma,0}(x)\leq c \ell(x,u)$. 
\item For any  $x\in\R^{n_x}$, there exists an infinite-length sequence of admissible inputs $\mathbf{u}_{s}(x)$ verifying\footnote{We recall that $\phi(k;x,\mathbf{u}_s \vert_k(x))$ is the solution to (\ref{eq:sys}) at time $k$, initialized at $x$ at time $0$, with inputs $\mathbf{u}_{s}(x)$, see Table \ref{tab:solutions}.} $\ell(\phi(k;x,\mathbf{u}_s(x) \vert_k),u_{s,k}(x))\leq \alpha_s(\sigma(x))\nu^k$ for any $k\in\Zo$. 
\end{enumerate}
Then SA\ref{sass:stabilizability} holds with $\overline\alpha=(1-\nu)^{-1}c\alpha_s$. 
\end{lem}

\begin{proof} Let $x\in\R^{n_x}$ and $\gamma\in(0,1]$, $V_{\gamma,\star}(x)\leq V_{1,\star}(x)\leq J_{{1}}(x,\mathbf{u}_s(x))\leq\sum_{k=0}^{\infty}\nu^k\alpha_s(\sigma(x))=(1-\nu)^{-1}\alpha_s(\sigma(x)){\leq (1-\nu)^{-1}c \alpha_s(\sigma(x))}$ as $\nu\in(0,1)$ {and $c\geq 1$}. Similarly, we derive from (\ref{eq:vi-as-finite-horizon}) and the imposed assumptions that $V_{\gamma,i}(x)\leq \sum_{k=0}^{i-1}\gamma^{k}\ell(\phi(k;x,\mathbf{u}_s(x)\vert_k),u_{s,k}(x))+\gamma^i V_{\gamma,0}(\phi(i;x,\mathbf{u}_s(x)\vert_i)) \leq \sum_{k=0}^{i-1}\ell(\phi(k;x,\mathbf{u}_s(x)\vert_k),u_{s,k}(x))+V_{\gamma,0}(\phi(i;x,\mathbf{u}_s(x)\vert_i)) \leq \sum_{k=0}^{i}c\ell(\phi(k;x,\mathbf{u}_s(x)\vert_k),u_{s,k}(x))\leq \sum_{k=0}^{i}c \nu^k \alpha_s(\sigma(x)) \leq (1-\nu)^{-1}c\alpha_s(\sigma(x))$. As $x$, $\gamma$ and $i$ have been arbitrarily selected, the desired result holds. 
\end{proof}

A consequence of SA\ref{sass:well-posedness}, SA\ref{sass:detectability} and SA\ref{sass:stabilizability} is that we can replace the infimum by a minimum in (\ref{eq:optimal-value-function}). 

\begin{lem}\label{lem:existence-optimal-inputs} For any $x\in\R^{n_x}$, $V_{\gamma,\star}(x)=\min_{\mathbf{u}} \sum_{k=0}^{\infty}\gamma^k\ell(\phi(k;x,\mathbf{u}\vert_k),u_k)$. 
\end{lem}

\begin{proof} The result is obtained by applying \cite[Thm. 1]{Keerthi-Gilbert-tac85} {under the notation of \cite{Keerthi-Gilbert-tac85}, $X=\R^{n_x}$, $U=\R^{n_u}$, $T_0=\{x\}$, and, for any $k\in\Zo$, $f_k=f$, $W_k=\mathcal{W}$ and $g_k=\gamma^k\ell$. Item a) of \cite[Thm. 1]{Keerthi-Gilbert-tac85} holds by SA\ref{sass:well-posedness}(ii), item b) by SA\ref{sass:well-posedness}(i), item c) by SA\ref{sass:well-posedness}(iii), and item d) is replaced by item d$_3$) of \cite[Thm. 2]{Keerthi-Gilbert-tac85} with $\phi_k=\gamma^k\alpha_\ell(|\cdot|)$ for any $k\in\Zo$, by SA\ref{sass:well-posedness}(iii).} 
Finally, to prove that item e) of \cite[Thm. 1]{Keerthi-Gilbert-tac85} is verified we proceed by contradiction. Suppose  item e) of \cite[Thm. 1]{Keerthi-Gilbert-tac85} does not hold, i.e., there exists $x\in\R^{n_x}$ such that for any infinite-length sequence of admissible inputs $\mathbf{u}$, $J_{\gamma}(x,\mathbf{u})=\infty$. This implies that $V_{\gamma,\star}(x)=\infty$, which contradicts SA\ref{sass:stabilizability}. Hence item e) of \cite[Thm. 1]{Keerthi-Gilbert-tac85} holds, and we therefore apply \cite[Thm. 1]{Keerthi-Gilbert-tac85} to obtain the desired result.     
\end{proof}

We finally make the next assumption on the comparison functions $\underline\alpha$ and $\overline\alpha$ appearing in SA\ref{sass:detectability} and SA\ref{sass:stabilizability}, respectively. 

\begin{sass}[\textbf{SA\ref{sass:optimal-policy-ugas}}]\label{sass:optimal-policy-ugas} There exist $\gamma_{\star}\in(0,1]$ and $\alpha\in\Kinf$ such that  $ \alpha(s)+(1-\gamma_\star)\overline\alpha(s)\leq \underline\alpha(s)$ for any $s\geq 0$, with $\underline\alpha,\overline\alpha$ from SA\ref{sass:detectability} and SA\ref{sass:stabilizability}, respectively. 
\end{sass}

SA\ref{sass:optimal-policy-ugas} guarantees that optimal policies generated using the set-valued map $H_{\gamma,\star}$ in (\ref{eq:optimal-policy}) ensure uniform global asymptotic stability properties\footnote{This follows from the proof of Lemma \ref{lem:bound-optimal-solutions} given in the appendix.} for system (\ref{eq:sys}) for $\gamma\in[\gamma_\star,1]$. We know that, in general, optimal policies exhibit semiglobal practical stabilizing properties in parameter $\gamma$ \cite{Postoyan-et-al-tac(optimal)}. We make SA\ref{sass:optimal-policy-ugas} to streamline the forthcoming results while still covering a wide range of systems and cost functions. For instance, SA\ref{sass:optimal-policy-ugas} holds in the undiscounted case by taking $\gamma_\star=1$ and $\alpha=\underline\alpha${, for which the results are new}. Another important class of problems for which SA\ref{sass:optimal-policy-ugas} holds is when the next assumption is satisfied. 

\begin{ass}\label{ass:linear-bounds} There exist $\overline{a},\underline{a}\in\Rlp$ {with $\underline a<\overline a$} such that $\overline \alpha(s)\leq \overline a s$ and $\underline\alpha(s)\geq \underline{a} s$ for any $s\geq 0$,  with $\underline\alpha,\overline\alpha$ from SA\ref{sass:detectability} and SA\ref{sass:stabilizability}, respectively.
\end{ass}

Indeed, SA\ref{sass:optimal-policy-ugas} holds under Assumption \ref{ass:linear-bounds} as shown next.

\begin{lem}\label{lem:satisfaction-of-sa-gamma-ugas}
When Assumption \ref{ass:linear-bounds} holds, SA\ref{sass:optimal-policy-ugas} is satisfied with any $\gamma_\star\in(1-\underline a/\overline a,1]$, $\alpha(s)=a s$  for any $s\geq 0$ and $a:=\underline a-\overline a+\gamma_\star\overline a$. 
\end{lem}

\begin{proof} Let $\gamma_\star\in (1-\underline a/\overline a,1]$. As $\overline a,\underline a >0$, $1-\underline a/\overline a<1$ and we note that $1-\underline a/\overline a> 0$, as $\overline a > \underline a$. 
Consider $\alpha$ as in Lemma \ref{lem:satisfaction-of-sa-gamma-ugas} and let $s\geq 0$, 
$\alpha(s)+(1-\gamma_\star)\overline\alpha(s) \leq (\underline a-\overline a+\gamma_\star\overline a) s + (1-\gamma_\star)\overline a s =  \underline a s \leq \underline\alpha(s)$. 
Hence $\alpha+(1-\gamma_\star)\overline\alpha\leq\underline\alpha$. We are left with proving that $\alpha\in\Kinf$, which is equivalent to proving $\underline a - \overline a +\gamma_\star\overline a>0$, which holds as  $\gamma_\star>1-\underline a/\overline a$.
\end{proof}

We highlight that Assumption \ref{ass:linear-bounds} is not a \emph{standing} assumption, it will only be invoked for some of the next results. 

Given SA\ref{sass:optimal-policy-ugas}, we restrict our attention to discount factors $\gamma$ taking values in $ [\gamma_\star, 1]$, and unless stated otherwise, all subsequent results apply to any such choice. {We also note that SA\ref{sass:detectability}-\ref{sass:optimal-policy-ugas} imply that $\alpha\leq\underline\alpha\leq\overline\alpha$.}


\subsection{Closed-loop systems}\label{subsect:closed-loop}

To conclude this section, we exploit the existence of minimizing input sequences guaranteed by the assumptions made so far to formalize what we mean by system (\ref{eq:sys}) being controlled either by a policy obtained with VI or by an optimal policy. 

Consider any iteration $i\in\Zp$, we write system (\ref{eq:sys}) controlled by any possible policies generated by VI at iteration $i$, which exists by Theorem \ref{th:recursive-feasibility}, as the next difference inclusion
\begin{eqn}\label{eq:sys-vi}
x(k+1) & \in & f(x,H_{\gamma,i}(x)) =: F_{\gamma,i}(x),
\end{eqn} 
and we denote by $\phi_{\gamma,i}(k;x)$ a (non-unique) solution to (\ref{eq:sys-vi}) at time $k$, initialized at $x$ at time $0$; $\phi_{\gamma,i}(\cdot;x)$ is a discrete arc defined on $\dom\phi_{\gamma,i}(\cdot;x)=\Zo$. 

We also write system (\ref{eq:sys}) controlled by optimal policies as 
\begin{eqn}\label{eq:sys-optimal}
x(k+1) & \in & f(x,H_{\gamma,\star}(x)) =: F_{\star,\gamma}(x),
\end{eqn}
where 
\begin{eqn}\label{eq:optimal-policy}
H_{\gamma,\star}(x): \R^{n_x} & \rightrightarrows & \R^{n_u} \\
x & \mapsto &   \argmin\limits_{u\in\mathcal{U}(x)}(\ell(x,u)+\gamma V_{\gamma,\star}(f(x,u))).
\end{eqn}
Note that $H_{\gamma,\star}$ takes non-empty values by Lemma \ref{lem:existence-optimal-inputs}. We denote by $\phi_{\gamma,\star}(k;x)$ a solution to (\ref{eq:sys-optimal}) at time $k\in\Zo$, initialized at $x$ at time $0$. Recall that a summary of the notation used to denote solutions for the various considered dynamical systems in this work is given in Table \ref{tab:solutions}.

\section{Potential stopping criteria}\label{sect:stopping-criteria}

We define the criterion used to stop VI in (\ref{eq:vi-value}) and formalize what is meant by  final iteration, final policies and final value function (Section \ref{subsect:class-stopping-criteria}). We then present design conditions for the stopping criterion and provide examples verifying these (Section \ref{subsect:conditions-c-stopvi}).

\begin{table*}[t!]
\centering
\renewcommand{\arraystretch}{1.2}
\scalebox{1}{\begin{tabular}{lcccccc} 
    Function $c_\stopvi$  &  Condition \ref{cond:finite-iteration}(i-a) & Condition \ref{cond:finite-iteration}(i-b) & Condition \ref{cond:finite-iteration}(ii-a) & Condition \ref{cond:finite-iteration}(ii-b) & Condition \ref{cond:stability}  \\
    \midrule
    $\theta $ & \bleucit{\ding{52}} & \bleucit{\ding{52}} & \textcolor{magenta}{\ding{55}} & \textcolor{magenta}{\ding{55}}& \bleucit{\ding{52}}\\
    $\gamma^i\theta$ & \textcolor{magenta}{\ding{55}} & \bleucit{\ding{52}} & \textcolor{magenta}{\ding{55}} & \textcolor{magenta}{\ding{55}} & \bleucit{\ding{52}}\\
    $\theta\sigma(x)$  & \textcolor{magenta}{\ding{55}} & \textcolor{magenta}{\ding{55}} & \bleucit{\ding{52}}  & \bleucit{\ding{52}} & \bleucit{\ding{52}}   \\
    $\gamma^i\theta\sigma(x)$  & \textcolor{magenta}{\ding{55}} & \textcolor{magenta}{\ding{55}} & \textcolor{magenta}{\ding{55}}  & \bleucit{\ding{52}} & \bleucit{\ding{52}}   \\
    $\theta V_{\gamma,i}(x)$  &  \textcolor{magenta}{\ding{55}} & \textcolor{magenta}{\ding{55}} & \bleucit{\ding{52}}$^\dagger$ & \bleucit{\ding{52}}$^\dagger$ & \bleucit{\ding{52}}  \\
    $\gamma^i \theta V_{\gamma,i}(x)$  &  \textcolor{magenta}{\ding{55}} &  \textcolor{magenta}{\ding{55}} &  \textcolor{magenta}{\ding{55}} & \bleucit{\ding{52}}$^\dagger$ &   \bleucit{\ding{52}}  \\
    \midrule
\end{tabular}}
\caption{Examples of stopping criteria  where  $\Theta=\Rlp$, symbols \bleucit{\ding{52}} (\textcolor{magenta}{\ding{55}}) mean that $c_\stopvi$ satisfies (or not) the condition.}
\label{tab:examples-stopping-criterion}
\end{table*}


\subsection{Definitions}\label{subsect:class-stopping-criteria}

We propose to stop VI at iteration $i+1$, where $i\in\Zp$ is the first iteration such that
\begin{eqn}\label{eq:stopping-criterion}
|V_{\gamma,i+1}(x)-V_{\gamma,i}(x)| & \leq & c_{\stopvi}(x,\theta,\gamma,i) & &  \forall x\in\mathcal{X},\tag{\textsc{stop}}
\end{eqn}
where $c_{\stopvi}:\mathcal{C}\to\Rlo$ is a function designed by the user, $\mathcal{C}:=\R^{n_x}\times\Theta\times(0,1]\times\Zo$, and $\mathcal{X}\subset\R^{n_x}$ is the region of the state space of interest. {Since (\ref{eq:stopping-criterion}) is only checked on $\mathcal{X}$, the properties derived in the sequel for the closed-loop system (\ref{eq:sys-vi}) at the final iteration only hold as long as its solutions remain in $\mathcal{X}$. Conditions ensuring this forward invariance property of $\mathcal{X}$ are given in Section \ref{subsect:closed-loop-embedding}}. Function $c_\stopvi$  is parameterized by $\theta\in\Theta\subset\R^{n_\theta}$ with $n_\theta\in\Zo$. We also allow $c_\stopvi$ to depend on the number of iterations $i$ and the discount factor $\gamma$ for the sake of generality. We can now clarify what we mean by final iteration, final policy and final value function.

\begin{defn} Given $\theta\in\Theta$ and $\gamma\in(0,1]$, the \emph{final iteration}  for VI as in (\ref{eq:vi-value}) under the stopping criterion (\ref{eq:stopping-criterion}), {denoted $i_\stopvi$}, is ${i_\stopvi}=\min\limits\{i'{\geq 2}\,:\,\forall x\in\mathcal{X},\,\,|V_{\gamma,i'}(x)-V_{\gamma,i'-1}(x)| \leq  c_{\stopvi}(x,\theta,\gamma,i'-1)\}${, provided this set is non-empty}. Given the final iteration {$i_\stopvi$}, a \emph{final policy} is any selection of $H_{\gamma,{i_\stopvi}}$ and the \emph{final value function} is $V_{\gamma,{i_\stopvi}}$.
\end{defn}

{\renewcommand{\thefootnote}{\dag}%
\footnotetext{Under the assumption that $\underline\alpha(s)\geq \underline a s$ for any $s\geq 0$, with $\underline a\in\Rlp$ and $\underline\alpha$ from SA\ref{sass:detectability}.}}


{Note that $i_\stopvi\geq 2$ by definition, i.e., at least two iterations are performed. This ensures that $V_{\gamma,i_\stopvi-1}$, with $i_\stopvi-1\geq1$, can be used as a Lyapunov function in Section \ref{sect:stability}, so that the results stated for $i\geq 2$  apply to the final iteration.} We do not know a priori whether (\ref{eq:stopping-criterion}) will be eventually satisfied, i.e., whether there exists a finite final iteration: we specify design conditions for $c_\stopvi$ for this purpose in the sequel.

\subsection{Conditions on $c_\stopvi$}\label{subsect:conditions-c-stopvi}

We present two sets of design conditions on $c_\stopvi$ depending on the property they induce. We emphasize that these are conditions and not assumptions as $c_\stopvi$ is synthesized by the user, so that these conditions can always be guaranteed by design and are thus made with no loss of generality.  We  provide examples of functions $c_\stopvi$ that satisfy each of these conditions in the following as well as in Table \ref{tab:examples-stopping-criterion}. Table \ref{tab:examples-stopping-criterion}  is not exhaustive, other choices for $c_\stopvi$ are possible.

The next condition is used in Section \ref{sect:finite-iterations} to ensure that the stopping criterion (\ref{eq:stopping-criterion}) is verified after a finite number of iterations.

\begin{cond}\label{cond:finite-iteration}
For any $\theta\in\Theta$, there exists $\varepsilon_\theta>0$ such that one of the next conditions holds  for any $i\in{\Zp}$, any $\gamma\in(0,1]$ and any $x\in\R^{n_x}$.
\begin{enumerate}
\item[(i-a)]  $c_\stopvi(x,\theta,\gamma,i)\geq \varepsilon_\theta$.
\item[(i-b)]  $c_\stopvi(x,\theta,\gamma,i)\geq \gamma^i \varepsilon_\theta$.
\item[(ii-a)] $c_\stopvi(x,\theta,\gamma,i)\geq  \varepsilon_\theta\sigma(x)$.
\item[(ii-b)] $c_\stopvi(x,\theta,\gamma,i)\geq  \gamma^i \varepsilon_\theta\sigma(x)$.
\end{enumerate}
\end{cond}

Condition \ref{cond:finite-iteration}(i-a) is for instance satisfied by the common choice $c_{\stopvi}(x,\theta,\gamma,i)=\theta$  by taking $\theta\in\Theta=\Rlp$ in which case $\varepsilon_\theta=\theta$, see, e.g., \cite{al-tamimi-tsmc08,wei-liu-tc2015,heydari2014revisiting} and \cite[Chap. 4.4]{SuttonBarto:98},  \cite[p.161]{puterman2014markov}, \cite{kvretinsky2023stopping}  in  stochastic settings. Another possible choice is $c_{\stopvi}(x,\theta,\gamma,i)=\tilde{c}(\sigma(x))+\theta$ with $\tilde{c}:\Rlo\to\Rlo$ and $\theta\in\Theta=\Rlp$. Similarly, Condition \ref{cond:finite-iteration}(i-b) is satisfied with $c_{\stopvi}(x,\theta,\gamma,i)=\tilde{c}(\sigma(x))+\gamma^i\theta$ with $\tilde{c}:\Rlo\to\Rlo$, $\theta\in\Theta=\Rlp$ and $\varepsilon_\theta=\theta$. Condition \ref{cond:finite-iteration}(ii-a), on the other hand, is ensured by taking $c_\stopvi(x,\theta,\gamma,i)=\theta\sigma(x)$ with $\theta\in\Theta=\Rlp$. Another example of stopping criterion verifying Condition \ref{cond:finite-iteration}(ii-a) is $c_\stopvi(x,\theta,\gamma,i)=\theta V_{\gamma,i}(x)$ when $\underline\alpha$ in SA\ref{sass:detectability} verifies $\underline\alpha(s)\geq \underline a s$ for any $s\in\Rlo$ with some $\underline a\in\Rlp$, as in Assumption \ref{ass:linear-bounds}. In this case,  $V_{\gamma,i}(x)\geq\underline\alpha(\sigma(x))\geq \underline a\sigma(x)$ for any $x\in\R^{n_x}$ and Condition \ref{cond:finite-iteration}(ii-a) holds with $\varepsilon_\theta=\underline a\theta$. We can similarly consider  $c_\stopvi(x,\theta,\gamma,i)=\gamma^i \theta\sigma(x)$ or $c_\stopvi(x,\theta,\gamma,i)=\gamma^i \theta V_{\gamma,i}(x)$ with $\theta\in\Theta\subset\Rlp$ for Condition \ref{cond:finite-iteration}(ii-b) to hold.



The next condition on $c_\stopvi$ is imposed to ensure stability properties for system (\ref{eq:sys-vi}) at iterations verifying (\ref{eq:stopping-criterion}), as shown in Section \ref{sect:stability}.

\begin{cond}\label{cond:stability} There exists $\zeta:\Rlo\times\Rlo\to\Rlo$ with $\zeta(s,\cdot)\in\K$ and $\zeta(\cdot,s)$ continuous, non-decreasing for any $s>0$ such that $c_\stopvi(x,\theta,\gamma,i)\leq\zeta(\sigma(x),|\theta|)$ for any $x\in\R^{n_x}$, $\theta\in\Theta$, $\gamma\in(0,1]$ and $i\in\Zo$. 
\end{cond}

Here as well the common choice $c_{\stopvi}(x,\theta,\gamma,i)=\theta$ with $\theta\in\Theta=\Rlp$ satisfies Condition \ref{cond:stability} by defining $\zeta(s_1,s_2)=s_2$ for any $s_1,s_2\in\Rlo$. Similarly, taking  $c_{\stopvi}(x,\theta,\gamma,i)=\tilde{c}(\sigma(x))\theta$ with $\tilde{c}:\Rlo\to\Rlo$ continuous, non-decreasing and $\theta\in\Theta=\Rlp$ leads to the satisfaction of Condition \ref{cond:stability} by defining $\zeta(s_1,s_2)=\tilde c(s_1) s_2$ for any $s_1,s_2\in\Rlo$. Another possible choice is $c_{\stopvi}(x,\theta,\gamma,i)=\theta V_{\gamma,i}(x)$ as 
$\theta V_{\gamma,i}(x)\leq\theta\overline\alpha(\sigma(x))=\zeta(\sigma(x),\theta)$ by SA\ref{sass:stabilizability}, for any $x\in\R^{n_x}$. 

We will show in Section \ref{sect:near-optimality} that the above conditions allow deriving novel near-optimality bounds on VI.

\section{Finite number of iterations}\label{sect:finite-iterations}



We establish in this section that VI with stopping criterion (\ref{eq:stopping-criterion}) always terminates after a finite number of iterations when $c_\stopvi$ in (\ref{eq:stopping-criterion}) is designed to satisfy Condition \ref{cond:finite-iteration}(i) (Section \ref{subsect:finite-iteration-condition-1-i}) for suitable sets $\mathcal{X}$. We then address the case where Condition \ref{cond:finite-iteration}(ii) holds and show that a final iteration exists for $\mathcal{X}=\R^{n_x}$ (Section \ref{subsect:finite-iteration-condition-1-ii}). 

\subsection{{Under Condition \ref{cond:finite-iteration}(i)}}\label{subsect:finite-iteration-condition-1-i}


Under Condition \ref{cond:finite-iteration}(i-a), the minimal number of iterations required to satisfy (\ref{eq:stopping-criterion}) is uniform over sets of initial conditions of the form $\{z\,:\,\sigma(z)\leq \Delta\}$ with $\Delta>0$ as formalized in the next theorem.

\begin{thm}\label{th:finite-iteration-condition-varepsilon} Suppose  Condition \ref{cond:finite-iteration}(i-a) holds. For any $\theta\in\Theta$ and any $\Delta>0$, there exists\footnote{$i_\text{bound}=\min\limits\left\{i\in\Zp\,:\,\max_{s\in[0,\Delta]}\left({\overline\alpha}(\xi(s,i)) +  {\overline\alpha}(\xi(s,i-1))\right)\leq\varepsilon_\theta\right\}$  with $\xi(s,k)=\underline\alpha^{-1}\circ(\id-\alpha\circ\overline\alpha^{-1})^{(k)}\circ\overline\alpha(s)$ for any $k\in\Zo$ and $s\geq 0$ as in Lemma \ref{lem:bound-optimal-solutions}.} $i_\text{bound}\in\Zp$ such that for any $\gamma\in[\gamma_\star,1]$ with $\gamma_\star$ from SA\ref{sass:optimal-policy-ugas}, any $x\in\R^{n_x}$ with $\sigma(x)\leq \Delta$ and any $i\geq i_{{\text{bound}}}$, $|V_{\gamma,i+1}(x)- V_{\gamma,i}(x)|\leq c_\stopvi(x,\theta,\gamma,i)$.
\end{thm}

\begin{proof} Let $i\in\Zp$, $\gamma\in[\gamma_\star,1]$ and $x\in\R^{n_x}$. By Lemma \ref{lem:bound-V-star-Vi} in the appendix,
\begin{eqn}
V_{\gamma,i+1}(x) - V_{\gamma,i}(x)  & \leq &  V_{\gamma,\star}(x)+\overline\alpha\left(\xi(\sigma(x),i)\right)  \\
& & - V_{\gamma,\star}(x) + \overline\alpha\left(\xi(\sigma(x),i-1)\right) \\
& = & \overline\alpha\left(\xi(\sigma(x),i)\right) + \overline\alpha\left(\xi(\sigma(x),i-1)\right)\!\!,
\end{eqn}
where  $\xi:(s,k)\mapsto\underline\alpha^{-1}\circ(\id-\alpha\circ\overline\alpha^{-1})^{(k)}\circ\overline\alpha(s)$ as in Lemma \ref{lem:bound-optimal-solutions}. We similarly obtain
\begin{eqn}
V_{\gamma,i}(x) - V_{\gamma,i+1}(x) \leq 
 \overline\alpha\left(\xi(\sigma(x),i)\right) + \overline\alpha\left(\xi(\sigma(x),i-1)\right).
\end{eqn}
Consequently,
\begin{eqn}
|V_{\gamma,i+1}(x) - V_{\gamma,i}(x)|  \leq 
 \overline\alpha\left(\xi(\sigma(x),i)\right) + \overline\alpha\left(\xi(\sigma(x),i-1)\right).
\end{eqn}
Consider now $\Delta>0$ and $x$ such that $\sigma(x)\leq\Delta$. We derive from the last inequality, as $\xi$ is continuous in its first argument by Lemma \ref{lem:bound-optimal-solutions} and $\overline\alpha$ is continuous being of class-$\Kinf$,
\begin{equation}\label{eq:proof-intermediate-bound-Vi+1-Vi}
\begin{array}{lllll}
|V_{\gamma,i+1}(x) - V_{\gamma,i}(x)| \\
\hspace{1cm}\leq \max_{s\in[0,\Delta]}\left(\overline\alpha\left(\xi(s,i)\right) + \overline\alpha\left(\xi(s,i-1)\right)\right).
\end{array}
\end{equation}
Let $\theta\in\Theta$ and $\varepsilon_\theta>0$ be as in Condition \ref{cond:finite-iteration}(i-a). As $\xi$ is non-increasing and converging to $0$ in its second argument, there exists $i_\text{bound}\in\Zp$ such that for any $i\geq i_\text{bound}$, $\max_{s\in[0,\Delta]}\left({\overline\alpha}(\xi(s,i)) + {\overline\alpha}(\xi(s,i-1))\right)\leq\varepsilon_\theta$. Note that $i_\text{bound}$ is independent of $\gamma$ as  $\xi$ is independent of $\gamma$. Consequently, for any $i\geq i_\text{bound}$ and any $\gamma\in[\gamma_\star,1]$,
\begin{eqn}
|V_{\gamma,i+1}(x) - V_{\gamma,i}(x)| & \leq & \varepsilon_\theta
\end{eqn}
and by Condition \ref{cond:finite-iteration}(i-a),
$|V_{\gamma,i+1}(x) - V_{\gamma,i}(x)| \leq c_\stopvi(x,\theta,\gamma,i)$. 
We have proved the desired result.
\end{proof}

{When $\mathcal{X}\subset\{z\,:\,\sigma(z)\leq\Delta\}$, the final iteration exists and
satisfies $i_{\stopvi}\leq\max\{2,i_{\text{bound}}+1\}$.} Interestingly, Theorem \ref{th:finite-iteration-condition-varepsilon} guarantees that the stopping criterion is eventually always satisfied for any iteration $i\geq i_{\text{bound}}$. 
We also have the next result regarding Condition \ref{cond:finite-iteration}(i-b) under extra assumptions.

\begin{thm}\label{th:finite-iteration-condition-varepsilon-gamma} Suppose Condition \ref{cond:finite-iteration}(i-b) and Assumption \ref{ass:linear-bounds} hold. For any $\theta\in\Theta$, any $\Delta>0$, any  $x\in\R^{n_x}$ with $\sigma(x)\leq \Delta$, any {$\gamma\in[\gamma_\star,1]$ with $\gamma>1-\overline a^{-1}a$,}  $a=\underline a-\overline a+\gamma_\star\overline a$ and $\gamma_\star\in(1-\underline a/\overline a,1]$ as in Lemma \ref{lem:satisfaction-of-sa-gamma-ugas}, and any $i\geq i_\text{bound}$ {with $i_\text{bound}=\max\{1,\left\lceil  \frac{\log(\varepsilon_\theta) + \log(1-\overline a^{-1} a) - \log(2\overline a^2 \underline a^{-1}\Delta)}{\log(1-\overline a^{-1} a)-\log(\gamma)} \right\rceil\}$ with $\varepsilon_\theta$ from Condition \ref{cond:finite-iteration}(i-b),} 
$|V_{\gamma,i+1}(x)- V_{\gamma,i}(x)|\leq c_\stopvi(x,\theta,\gamma,i)$.
\end{thm}

\begin{proof} Let $i\in\Zp$, $\gamma\in{[\gamma_\star,1]}$ with $\gamma>1-\overline a^{-1}a$, $\gamma_\star$ as in Lemma \ref{lem:satisfaction-of-sa-gamma-ugas} and $x\in\R^{n_x}$. {As $\gamma_\star\leq 1$ and $\underline a < \overline a$,  $1-\overline a^{-1}a> 0$. Furthermore, $a>0$ as $\gamma_\star>1-\nicefrac{\underline a}{\overline a}$, hence $1-\overline a^{-1}a<1$.} 
We first derive an upper-bound on $\xi$. Let $s\in\Rlo$ and $k\in\Zo$, $\xi(s,k)=\underline\alpha^{-1}\circ(\id-\alpha\circ\overline\alpha^{-1})^{(k)}\circ\overline\alpha(s)$ and, given the made assumptions, $\underline\alpha^{-1}(s)\leq\underline a^{-1} s$, $\overline\alpha(s)\leq\overline a s$. We can take $\alpha(s)=a s$  for any $s\geq 0$ by Lemma \ref{lem:satisfaction-of-sa-gamma-ugas}. Consequently, 
\begin{eqn}\label{eq:proof-thm-bound-xi}
\xi(s,k) & \leq & \overline a\underline a^{-1} (1-\overline a^{-1}a)^k s.
\end{eqn}
We follow the same steps as in the proof of Theorem \ref{th:finite-iteration-condition-varepsilon}, we obtain (\ref{eq:proof-intermediate-bound-Vi+1-Vi}) holds and therefore, given (\ref{eq:proof-thm-bound-xi}) and the fact that 
 $\overline\alpha(\xi)\leq \overline a \xi$,
$|V_{\gamma,i+1}(x) - V_{\gamma,i}(x)|  \leq \max_{s\in[0,\Delta]}\left(\overline\alpha\left(\xi(s,i)\right) + \overline\alpha\left(\xi(s,i-1)\right)\right) 
 \leq  \max_{s\in[0,\Delta]}\left({\overline a^2\underline a^{-1} (1-\overline a^{-1}a)^{i} s + \overline a^2\underline a^{-1} (1-\overline a^{-1}a)^{i-1} s}\right)
 $ ${ \leq } $ ${2\overline a^2\underline a^{-1} (1-\overline a^{-1}a)^{i-1} \Delta}$, 
where we have used the fact that $1-\overline a^{-1}a\in[0,1)$.

Let $\theta\in\Theta$, $\varepsilon_\theta>0$ be as in Condition \ref{cond:finite-iteration}(i-b) {
and $i_\text{bound}={\max\{1,\left\lceil  \frac{\log(\varepsilon_\theta) + \log(1-\overline a^{-1} a)- \log(2\overline a^2\underline a^{-1}\Delta)}{\log(1-\overline a^{-1} a)-\log(\gamma)} \right\rceil\}}$. {For any $i\geq i_{\text{bound}}$,  ${2\overline a^2/\underline a (1-\overline a^{-1}a)^{i-1} \Delta}\leq \gamma^{i}\varepsilon_\theta$.} {Note that this last expression for $i_{\text{bound}}$} is well-defined as $1-\overline a^{-1}a \in(0,1)$ and $\gamma>1-\overline a^{-1}a$.} 
{Since $(1-\overline a^{-1}a)\gamma^{-1}<1$, $((1-\overline a^{-1}a)\gamma^{-1})^i$ is non-increasing in $i$ and}  for any $i\geq i_\text{bound}$, $
|V_{\gamma,i+1}(x) - V_{\gamma,i}(x)|  \leq  \gamma^i  \varepsilon_\theta$,
and the desired result holds. 
\end{proof}

Theorem \ref{th:finite-iteration-condition-varepsilon-gamma} imposes an extra condition on the discount factor $\gamma$ compared to Theorem \ref{th:finite-iteration-condition-varepsilon}, which is justified by the fact that the considered stopping criterion is more demanding in the sense that it exponentially converges to $0$ with the number of iterations. 

\subsection{Under Condition \ref{cond:finite-iteration}(ii)}\label{subsect:finite-iteration-condition-1-ii}

Under Condition \ref{cond:finite-iteration}(ii-a), the minimal number of iterations for  which (\ref{eq:stopping-criterion}) holds is independent of the set where $x$ is taken as stated next, under Assumption \ref{ass:linear-bounds}.

\begin{thm}\label{th:finite-iteration-condition-L-sigma(x)} Suppose the following holds.
\begin{enumerate}[label=(\roman*)]
\item Assumption \ref{ass:linear-bounds} is satisfied.
\item Condition \ref{cond:finite-iteration}(ii-a) is satisfied.
\end{enumerate}
For any $\theta\in\Theta$, any $x\in\R^{n_x}$, any $\gamma\in[\gamma_\star,1]$ with $\gamma_\star$ {as in Lemma \ref{lem:satisfaction-of-sa-gamma-ugas}} and any $i\geq i_\text{bound}$ {with $i_{\text{bound}}={\max\{1,\left\lceil  \frac{\log(\varepsilon_\theta) + \log(1-\overline a^{-1} a)- \log(2\overline a^2/{\underline a})}{\log(1-\overline a^{-1} a)} \right\rceil\}}$, 
where $a=\underline a -\overline a +\gamma_\star \overline a$}, $|V_{\gamma,i+1}(x)- V_{\gamma,i}(x)|\leq c_\stopvi(x,\theta,\gamma,i)$.
\end{thm}

\begin{proof} 
Let $i\in\Zp$, $\gamma\in[\gamma_\star,1]$ and $x\in\R^{n_x}$. By following similar lines as in the proof of Theorem \ref{th:finite-iteration-condition-varepsilon} and (\ref{eq:proof-thm-bound-xi}), we obtain
\begin{eqn}
{|V_{\gamma,i+1}(x) - V_{\gamma,i}(x)| \leq 2\overline a^2/\underline a (1-\overline a^{-1}a)^{i-1}\sigma(x)} 
\end{eqn}
{Let 
$i_\text{bound}=\max\{1,\left\lceil  \frac{\log(\varepsilon_\theta) + \log(1-\overline a^{-1} a)- \log(2\overline a^2/{\underline a})}{\log(1-\overline a^{-1} a)} \right\rceil\}$, which is such that  $2\overline a^2/\underline a (1-\overline a^{-1}a)^{i_{\text{bound}}-1}\leq \varepsilon_\theta$ with $\varepsilon_\theta$ as in Condition \ref{cond:finite-iteration}(ii-a); such an integer is well-defined as $1- \overline{a}^{-1}a\in(0,1)$ as shown in the proof of Theorem \ref{th:finite-iteration-condition-varepsilon-gamma}}, and $\varepsilon_\theta>0$. Then, {as $1-\overline a^{-1}a<1$, $(1-\overline a^{-1}a)^{i-1}$ is non-increasing in $i$, and} for any $i\geq i_\text{bound}$,
$|V_{\gamma,i+1}(x) - V_{\gamma,i}(x)|  \leq  \varepsilon_\theta \sigma(x)$ 
and by Condition \ref{cond:finite-iteration}(ii-a), 
\begin{eqn}
|V_{\gamma,i+1}(x) - V_{\gamma,i}(x)| & \leq & c_\stopvi(x,\theta,\gamma,i).
\end{eqn}
We have proved the desired result.
\end{proof}

Under the conditions of Theorem \ref{th:finite-iteration-condition-L-sigma(x)}, we can consider $\mathcal{X}$ in (\ref{eq:stopping-criterion}) to be any subset of $\R^{n_x}$, including $\R^{n_x}$ itself, and we are guaranteed that there exists a finite iteration from which (\ref{eq:stopping-criterion}) holds.

The next theorem addresses the case of Condition \ref{cond:finite-iteration}(ii-b). Its proof is omitted as it follows very similarly to  the proofs of Theorems \ref{th:finite-iteration-condition-varepsilon-gamma} and \ref{th:finite-iteration-condition-L-sigma(x)}.

\begin{thm}\label{th:finite-iteration-condition-L-sigma(x)-gamma} Suppose the following holds.
\begin{enumerate}[label=(\roman*)]
\item Assumption \ref{ass:linear-bounds} is satisfied.
\item  Condition \ref{cond:finite-iteration}(ii-b) is satisfied.
\end{enumerate}
For any $\theta\in\Theta$,  any $x\in\R^{n_x}$, any {$\gamma\in[\gamma_\star,1]$ with $\gamma>1-\overline a^{-1}a$,}  $\gamma_\star$ {as in Lemma \ref{lem:satisfaction-of-sa-gamma-ugas}} and any $i\geq i_\text{bound}$ {with $i_{\text{bound}}={\max\{1,\left\lceil  \frac{\log(\varepsilon_\theta) + \log(1-\overline a^{-1} a)- \log(2\overline a^2/{\underline a})}{\log(1-\overline a^{-1} a)-\log(\gamma)} \right\rceil\}}$, 
where $a=\underline a -\overline a +\gamma_\star \overline a$ as in Theorem \ref{th:finite-iteration-condition-varepsilon-gamma}}, $|V_{\gamma,i+1}(x)- V_{\gamma,i}(x)|\leq c_\stopvi(x,\theta,\gamma,i)$.
\end{thm}

The estimates of $i_{{\text{bound}}}$ in Theorems \ref{th:finite-iteration-condition-varepsilon}-\ref{th:finite-iteration-condition-L-sigma(x)-gamma}  are typically subject to some conservatism and do not need to be explicitly computed to implement VI with a stopping criterion.

\section{Stability guarantees}\label{sect:stability}


The aim of this section is to derive stability properties for system (\ref{eq:sys-vi}) at the final iteration. In fact, we will show that system (\ref{eq:sys-vi}) exhibits stability properties at \emph{any} iteration satisfying (\ref{eq:stopping-criterion}),  as this follows without additional effort. To ease the stability analysis, we embed  system (\ref{eq:sys-vi}) at any iteration verifying (\ref{eq:stopping-criterion}) into a ``larger'' auxiliary dynamical system  (Section \ref{subsect:closed-loop-embedding}) with  the guarantee that any stability property for this auxiliary system applies to the original one. After we derive Lyapunov properties, we will establish that the embedded system exhibits  semiglobal practical stability in $\theta$  (Section \ref{subsect:sgpas}). Stronger stability properties are deduced by imposing extra conditions on the stopping criterion and the functions  in SA\ref{sass:detectability} and SA\ref{sass:stabilizability} as in Assumption \ref{ass:linear-bounds} (Section \ref{subsect:stronger-stability}).

\subsection{Closed-loop system at the final iteration and its embedding}\label{subsect:closed-loop-embedding}

Since the stopping criterion in (\ref{eq:stopping-criterion}) is required to hold on the set $\mathcal{X}$, which may not be necessarily the whole set $\R^{n_x}$, we shall focus on the next \emph{constrained} difference inclusion, given $\theta\in\Theta$ used to define the stopping criterion\footnote{We do not index the set-valued maps that depend on $\theta$ with $\theta$, namely $\widehat{H}_{\gamma,i}$ and $\widehat{F}_{\gamma,i}$, to avoid overloading the  notation.},
\begin{equation}\label{eq:sys-vi-stopping}
\begin{array}{llll}
x(k+1) \in f(x,H_{\gamma,i}(x)),\,\,x(k)\in\mathcal{X}, 
\end{array}
\end{equation}
{with $i$ s.t.  (\ref{eq:stopping-criterion}) holds at iteration $i-1$}.  
Given an initial condition $x$ necessarily in $\mathcal{X}$, we mean by a solution to (\ref{eq:sys-vi-stopping}) initialized at $x$ at time $0$, any discrete arc $\phi_{\gamma,i}(\cdot;x)$ verifying $\phi_{\gamma,i}(0;x)=x$,  $\phi_{\gamma,i}(k+1;x)\in f(\phi_{\gamma,i}(k;x),H_{\gamma,i}(\phi_{\gamma,i}(k;x))$ and $\phi_{\gamma,i}(k;x)\in\mathcal{X}$ for any $k\in\dom\phi_{\gamma,i}$ where we write $\dom\phi_{\gamma,i}$ for $\dom\phi_{\gamma,i}(\cdot;x)$ with  slight abuse of notation. Contrary to the other dynamical systems encountered in the previous sections, system (\ref{eq:sys-vi-stopping}) (and also system (\ref{eq:sys-embedding}) below) may generate maximal solutions that are non-complete as they may leave the set $\mathcal{X}$. For this reason, properties established for  such constrained systems will hold on the domain of the solutions, like in e.g., \cite{Goebel-Sanfelice-Teel-book,sanfelice-teel-aut10}, and we will comment in due time on the completeness of maximal solutions, or equivalently, on the forward invariance of set $\mathcal{X}$ for the considered system.

To ease the exposition of the stability analysis, we embed the above system for a given $\theta\in\Theta$ into the next constrained difference inclusion, for any $i\in\Zp$,
\begin{eqn}\label{eq:sys-embedding}
x(k+1)  \in f(x,\widehat{H}_{\gamma,i}(x))=:\widehat{F}_{\gamma,i}(x) & & x(k)\in\mathcal{X}, 
\end{eqn}
where
\begin{eqn}\label{eq:widehat-H-i}
\widehat{H}_{\gamma,i}(x) & := & \left\{u\in\mathcal{U}(x)\,:\,\ell(x,u)+\gamma V_{\gamma,i-1}(f(x,u)) \right. \\ & & \hspace{1.8cm}\left.\leq V_{\gamma,i-1}(x)+c_{\stopvi}(x,\theta,\gamma,{i-1})\right\}\!.
\end{eqn}
We denote by $\widehat\phi_{\gamma,i,\theta}$ the solutions to (\ref{eq:sys-embedding}) defined over the domain $\dom\widehat\phi_{\gamma,i,\theta}$ with the same slight abuse of notation as above. We  highlight that  (\ref{eq:sys-embedding}) defines a \emph{family of systems} parameterized by $i$, $\gamma$ and $\theta$.
Finally, we use the terminology embedding because any solution to (\ref{eq:sys-vi-stopping}) is also a solution to (\ref{eq:sys-embedding}) as shown next, but the opposite statement may not be true. 

\begin{table*}[t!]
\centering
\renewcommand{\arraystretch}{1.2} 
\begin{tabular}{ll}
System & Solution \\
\midrule
$x(k+1) = f(x(k),u(k))$ & $\phi(k;x,\mathbf{u}\vert_k)$ \\
$x(k+1) \in f(x,H_{\gamma,i}(x))=F_{\gamma,i}(x)$ with $H_{\gamma,i}$ in (\ref{eq:vi-policy}) & $\phi_{\gamma,i}(k;x)$ \\
$x(k+1) \in f(x,H_{\gamma,\star}(x)) = F_{\star,\gamma}(x)$ with $H_{\gamma,\star}$ in (\ref{eq:optimal-policy}) & $\phi_{\gamma,\star}(k;x)$ \\
$x(k+1)  \in f(x,\widehat{H}_{\gamma,i}(x))=\widehat{F}_{\gamma,i}(x)$, $ x(k)\in\mathcal{X}$ with $\widehat{H}_{\gamma,i}$ in (\ref{eq:widehat-H-i}) & $\widehat\phi_{\gamma,i,\theta}(k;x)$ \\
\midrule
\end{tabular}
\caption{Summary of the used notation to denote solutions at time $k$, given initial state $x$ at time $0$, and infinite-length sequence of admissible inputs $\mathbf{u}$ (when relevant) and discount factor $\gamma$.}
\label{tab:solutions}
\end{table*}

\begin{lem}\label{lem:embedding} Given any $\theta\in\Theta$, $\gamma\in(0,1]$ {and $i\geq 2$ such that (\ref{eq:stopping-criterion}) holds at iteration $i-1$}, any solution to (\ref{eq:sys-vi-stopping}) is also a solution to (\ref{eq:sys-embedding}).
\end{lem}

\begin{proof} Let $\theta\in\Theta$, $\gamma\in(0,1]$ {and $i\geq 2$ such that (\ref{eq:stopping-criterion}) holds at iteration $i-1$}, $x\in\mathcal{X}$ and $u\in H_{\gamma,i}(x)$.  By (\ref{eq:vi-value}), 
$V_{\gamma,i}(x) = \ell(x,u)+\gamma V_{\gamma,i-1}(f(x,u))$. 
As (\ref{eq:stopping-criterion}) holds, $V_{\gamma,i}(x)\leq V_{\gamma,i-1}(x)+c_\stopvi(x,\theta,\gamma,{i-1})$, consequently
$\ell(x,u)+\gamma V_{\gamma,i-1}(f(x,u)) \leq V_{\gamma,i-1}(x)+c_\stopvi(x,\theta,\gamma,{i-1})$, 
which means that $u\in\widehat{H}_{\gamma,i}(x)$ given the definition of $\widehat{H}_{\gamma,i}$ in (\ref{eq:widehat-H-i}). As $x$ and $u$ have been arbitrarily selected in $\mathcal{X}$ and $H_{\gamma,i}(x)$ respectively, we derive that $H_{\gamma,i}(\mathcal{X})\subset \widehat{H}_{\gamma,i}(\mathcal{X})$. Consequently, as all solutions to (\ref{eq:sys-vi-stopping}) and (\ref{eq:sys-embedding}) are constrained to take values in $\mathcal{X}$, any solution to (\ref{eq:sys-vi-stopping}) is also a solution to (\ref{eq:sys-embedding}).
\end{proof}

Given $\theta\in\Theta$, as $H_{\gamma,i}(x)\subset \widehat{H}_{\gamma,i}(x)$ for any $x\in\mathcal{X}$ and {iteration $i$ such that (\ref{eq:stopping-criterion}) holds at iteration $i-1$}, $\widehat{H}_{\gamma,i}$ is guaranteed to be non-empty under the conditions of Theorems \ref{th:finite-iteration-condition-varepsilon} and \ref{th:finite-iteration-condition-L-sigma(x)}. When (\ref{eq:stopping-criterion}) does not hold, $\widehat{H}_{\gamma,i}$ may have  the empty set as a possible value. 

\begin{rem} The reverse statement of Lemma \ref{lem:embedding} does not hold in general in the sense that a solution to (\ref{eq:sys-embedding}) with $\widehat H_{\gamma,i}(x)\neq\emptyset$ for any $x\in\R^{n_x}$, given $\theta\in\Theta$, is not necessarily a solution to (\ref{eq:sys-vi-stopping}). To see this consider  $u\in \widehat H_{\gamma,i}(x)$ for some $x\in\mathcal{X}$ and $\theta\in\Theta$. There is no reason for $u$ to be such that $V_{\gamma,i}(x)=\ell(x,u)+\gamma V_{\gamma,i-1}(f(x,u))$ and even if this would {be} the case this would imply that $V_{\gamma,i}(x)-V_{\gamma,i-1}(x)\leq c_\stopvi(x,\theta,\gamma,i-1)$ but not necessarily that $V_{\gamma,i-1}(x)-V_{\gamma,i}(x)\leq c_\stopvi(x,\theta,\gamma,i-1)$ as required for (\ref{eq:stopping-criterion}) to hold.
\end{rem}

Lemma \ref{lem:embedding} implies that any stability property established for system (\ref{eq:sys-embedding}) also holds for system (\ref{eq:sys-vi-stopping}), using the same reasoning as in \cite[Chap. 3.4]{Goebel-Sanfelice-Teel-book}. For this reason, we analyse the stability properties of system (\ref{eq:sys-embedding}) in the sequel. {In the terminology of \cite{Goebel-Sanfelice-Teel-book}, the stability properties established in the sequel for (\ref{eq:sys-vi-stopping}) and (\ref{eq:sys-embedding}) are thus \emph{pre-stability} properties because they hold on the domain of the solutions, which may be bounded; they become standard stability properties when maximal solutions are complete, see the explanations after Theorem \ref{th:stability-sgpas} and Proposition \ref{prop:forward-invariance}.}

\subsection{Semiglobal practical stability}\label{subsect:sgpas}

In the following, we establish stability properties for system (\ref{eq:sys-embedding}) by exploiting the next Lyapunov properties.

\begin{thm}\label{th:lyapunov} Given any $\theta\in\Theta$, $\gamma\in[\gamma_\star,1]$ with $\gamma_\star$ from SA\ref{sass:optimal-policy-ugas} and any $i\in\Zp$, the following holds.
\begin{enumerate}[label=(\roman*)]
\item For any $x\in\R^{n_x}$, $\underline\alpha(\sigma(x))\leq V_{\gamma,i}(x)\leq \overline\alpha(\sigma(x))$, with $\underline\alpha$ and $\overline\alpha$ from SA\ref{sass:detectability} and SA\ref{sass:stabilizability}.
\item For any $x\in\mathcal{X}$ and any $\widehat h_{\gamma,i+1}\in \widehat{H}_{\gamma,i+1}$, $V_{\gamma,i}(f(x,\widehat h_{\gamma,i+1}(x))-V_{\gamma,i}(x)\leq - \alpha(\sigma(x))+\frac{1}{\gamma_\star}c_\stopvi(x,\theta,\gamma,i)$ with $\alpha$ from SA\ref{sass:optimal-policy-ugas}.
\end{enumerate}
\end{thm}

\begin{proof} Let $\theta\in\Theta$,  $\gamma\in[\gamma_\star,1]$, $i\in\Zp$ and $x\in\R^{n_x}$. We have $\ell(x,h_{\gamma,i}(x))\leq V_{\gamma,i}(x)$ for any $h_{\gamma,i}\in  H_{\gamma,i}$ and, by SA\ref{sass:detectability}, $\ell(x,h_{\gamma,i}(x))\geq\underline\alpha(\sigma(x))$, therefore $V_{\gamma,i}(x)\geq \underline\alpha(\sigma(x))$. On the other hand, $V_{\gamma,i}(x)\leq\overline\alpha(\sigma(x))$ by SA\ref{sass:stabilizability}.  We have proved that Theorem \ref{th:lyapunov}(i) holds.

Let  $\widehat h_{\gamma,i+1}\in \widehat{H}_{\gamma,i+1}$. By the definition of $\widehat{H}_{\gamma,i+1}$ in (\ref{eq:widehat-H-i}), 
$V_{\gamma,i}(f(x,\widehat h_{\gamma,i+1}(x))) - V_{\gamma,i}(x) \leq - \frac{1}{\gamma}\ell(x,\widehat{h}_{\gamma,i+1}(x)) + \frac{1-\gamma}{\gamma}V_{\gamma,i}(x)+\frac{1}{\gamma}c_\stopvi(x,\theta,\gamma,i)$. By proceeding like in the beginning of the proof of Lemma \ref{lem:bound-optimal-solutions}, {and as $\gamma\geq \gamma_\star$,} we derive 
\begin{eqn}
V_{\gamma,i}(f(x,\hat{h}_{\gamma,i+1}(x)))  - V_{\gamma,i}(x) & \leq & -\alpha(\sigma(x)) \\
& & + \frac{1}{\gamma_\star}c_\stopvi(x,\theta,\gamma,i).
\end{eqn}
This completes the proof.
\end{proof}

Theorem \ref{th:lyapunov}(i) implies that each value function $V_{\gamma,i}$ is positive definite with respect to the set $\sigma^{-1}(0)$, in the sense that it takes positive values except in $\sigma^{-1}(0)$ where it vanishes.  Moreover $V_{\gamma,i}(x)$ converges to infinity as $\sigma(x)$ tends to infinity. On the other hand, Theorem \ref{th:lyapunov}(ii) establishes that the value function $V_{\gamma,i}$ satisfies a dissipation inequality along any solution to (\ref{eq:sys-embedding}) at iteration $i+1$. Note that $V_{\gamma,i}$ is considered in Theorem \ref{th:lyapunov} while the policies are taken in $\widehat H_{\gamma,i+1}$. Returning to the dissipativity property, $V_{\gamma,i}$ is guaranteed to strictly decrease outside $\sigma^{-1}(0)$ up to a perturbative term due to the stopping criterion, and the latter can be made as small as desired by tuning $\theta$ thanks to Condition \ref{cond:stability}. This leads to the next stability theorem, proved in Appendix \ref{appendix:proof-main-results}.

\begin{thm}\label{th:stability-sgpas}  Consider system (\ref{eq:sys-embedding}) with $c_\stopvi$ designed to satisfy  Condition \ref{cond:stability}. There exists $\beta\in\KL$ such that for any $\delta,\Delta>0$, {there exists $\overline\theta>0$ such that for any} $\gamma\in[\gamma_\star,1]$ with $\gamma_\star$ from SA\ref{sass:optimal-policy-ugas}, any $i\geq 2$,  any $x\in\mathcal{X}$ with $\sigma(x)\leq \Delta$ and $\theta\in\Theta$ with $|\theta|\leq\overline\theta$: any solution $\widehat\phi_{\gamma,i,\theta}$ satisfies  
$\sigma(\widehat\phi_{\gamma,i,\theta}(k;x)) \leq \max\{\beta(\sigma(x),k),\delta\}$
for all $k\in\dom\widehat\phi_{\gamma,i,\theta}$.
\end{thm}

Theorem \ref{th:stability-sgpas} establishes that system (\ref{eq:sys-embedding}) satisfies a semiglobal, practical stability property in the sense that for any set of initial conditions of the form $\{z\,:\,\sigma(z)\leq \Delta\}$, solutions to (\ref{eq:sys-embedding}) verify the $\KL$-property in Theorem \ref{th:stability-sgpas}, which guarantees the convergence of $\sigma$ along such solutions to $0$ up to an error given by $\delta$ that can be made as small as desired by taking $|\theta|$ sufficiently small. We also  emphasize that the stability property in Theorem \ref{th:stability-sgpas} only applies to solutions to (\ref{eq:sys-embedding}) that remain in $\mathcal{X}$ by definition. This is reflected by the fact that the stability property in Theorem \ref{th:stability-sgpas}  holds for any time $k$ in the domain of the solution. A natural question is under what conditions a solution initialized in $\mathcal{X}$ is guaranteed to remain in $\mathcal{X}$ for all positive times. When $\mathcal{X}=\R^{n_x}$, as in Theorem \ref{th:finite-iteration-condition-L-sigma(x)}, {and (\ref{eq:stopping-criterion}) holds at iteration $i-1$, this is the case, as $\widehat H_{\gamma,i}(z)\supset H_{\gamma,i}(z)\neq\emptyset$ for any $z\in\R^{n_x}$ by Theorem \ref{th:recursive-feasibility} and Lemma \ref{lem:embedding}, so that maximal solutions are complete.} Otherwise, we have the next result.

\begin{prop}\label{prop:forward-invariance} Consider system (\ref{eq:sys-embedding}) with $c_\stopvi$ that satisfies  Condition \ref{cond:stability}. Suppose there exists $\Delta>0$ such that $\{z\,:\,\sigma(z)\leq\Delta\}\subset\mathcal{X}$ and let $\overline\theta\in\Rlp$ be such that $\zeta(\Delta,\overline\theta) \leq {\gamma_\star\alpha\circ\overline\alpha^{-1}\circ\underline\alpha(\Delta)}$. Then, for any $\gamma\in[\gamma_\star,1]$ with $\gamma_\star$ from SA\ref{sass:optimal-policy-ugas}, any   $i\geq 2$ {such that (\ref{eq:stopping-criterion}) holds at iteration $i-1$} and $\theta\in\Theta$ with $|\theta|\leq\overline\theta$: the set $\{z\,:\, V_{\gamma,{i-1}}(z)\leq \underline\alpha(\Delta)\}$ is included in $\mathcal{X}$ and is strongly forward invariant, i.e., any solution initialized in this set remains in it for all positive times. Furthermore, any solution initialized in  $\{z\,:\,\sigma(z)\leq \overline\alpha^{-1}\circ\underline\alpha(\Delta)\}$ remains in $\mathcal{X}$ for all positive times.
\end{prop}
\begin{proof}
First, such a $\overline\theta$ exists as $\zeta$  is of class-$\K$ in its second argument. Let $\gamma\in[\gamma_\star,1]$, $i\geq 2$ {such that (\ref{eq:stopping-criterion}) holds at iteration $i-1$} and $|\theta|\leq\overline\theta$. We have $\{z\,:\, V_{\gamma,{i-1}}(z)\leq \underline\alpha(\Delta)\}\subset \{z\,:\, \sigma(z)\leq \Delta\}\subset\mathcal{X}$ as $\underline\alpha(\sigma(z))\leq V_{\gamma,{i-1}}(z)$ for any $z\in\R^{n_x}$ by Theorem \ref{th:lyapunov}(i). 
{We now show that this set is strongly forward invariant for system (\ref{eq:sys-embedding}). Let $z\in\R^{n_x}$ be such that $V_{\gamma,i-1}(z)\leq \underline\alpha(\Delta)$. Then $\underline\alpha(\sigma(z))\leq V_{\gamma,i-1}(z)\leq\underline\alpha(\Delta)$ by Theorem \ref{th:lyapunov}(i), hence  $z\in\mathcal{X}$. Let $u\in\widehat{H}_{\gamma,i}(z)$. By (\ref{eq:proof-stability-lyapunov-bounds}) in the appendix, $|\theta|\leq\overline\theta$, $\sigma(z)\leq\Delta$, the monotonicity properties of $\zeta$ in Condition \ref{cond:stability} and the fact that $\alpha(\sigma(z))\geq\alpha\circ\overline\alpha^{-1}(V_{\gamma,i-1}(z))$ by Theorem \ref{th:lyapunov}(i),
\begin{eqn}
V_{\gamma,i-1}(f(z,u)) & \leq & (\id-\alpha\circ\overline\alpha^{-1})(V_{\gamma,i-1}(z))+\frac{1}{\gamma_\star}\zeta(\Delta,\overline\theta)\\
& \leq & \underline\alpha(\Delta)-\alpha\circ\overline\alpha^{-1}\circ\underline\alpha(\Delta)+\frac{1}{\gamma_\star}\zeta(\Delta,\overline\theta)\\
& \leq & \underline\alpha(\Delta),
\end{eqn}
where the second inequality follows from the fact that $\id-\alpha\circ\overline\alpha^{-1}$ is non-decreasing{, which holds without loss of generality \cite[Remark 2]{Granzotto-et-al-tac-finite-discounted-horizon},} and $V_{\gamma,i-1}(z)\leq \underline\alpha(\Delta)$, and the last one from the condition imposed on $\overline\theta$, namely $\frac{1}{\gamma_\star}\zeta(\Delta,\overline\theta)\leq\alpha\circ\overline\alpha^{-1}\circ\underline\alpha(\Delta)$. Hence $V_{\gamma,i-1}(f(z,u))\leq\underline\alpha(\Delta)$ for any $u\in\widehat{H}_{\gamma,i}(z)$, i.e., the set $\{z\,:\,V_{\gamma,i-1}(z)\leq\underline\alpha(\Delta)\}$ is strongly forward invariant for system (\ref{eq:sys-embedding}).} 
On the other hand, $\{z\,:\,\sigma(z)\leq \overline\alpha^{-1}\circ\underline\alpha(\Delta)\}\subset \{z\,:\, V_{\gamma,{i-1}}(z)\leq \underline\alpha(\Delta)\}\subset\mathcal{X}$. 
Hence,  any solution to (\ref{eq:sys-embedding})  initialized in the set $\{z\,:\,\sigma(z)\leq \overline\alpha^{-1}\circ\underline\alpha(\Delta)\}$ remains in $\{z\,:\, V_{\gamma,i-1}(z)\leq \underline\alpha(\Delta)\}$ and thus in $\cal X$ for all positive times. {Finally, since (\ref{eq:stopping-criterion}) holds at iteration $i-1$, $\emptyset\neq H_{\gamma,i}(z)\subset\widehat H_{\gamma,i}(z)$ for any $z\in\mathcal X$ by Theorem \ref{th:recursive-feasibility} and Lemma \ref{lem:embedding}. Hence $\widehat F_{\gamma,i}(z)\neq\emptyset$ for any $z$ in $\{z:V_{\gamma,i-1}(z)\le\underline\alpha(\Delta)\}\subset\mathcal X$. Together with the invariance property shown above, this implies that every maximal solution of (\ref{eq:sys-embedding}) initialized in this set is complete. }
\end{proof}

{By taking $|\theta|$ smaller than both the constants $\overline\theta$ of Theorem \ref{th:stability-sgpas} and Proposition \ref{prop:forward-invariance},
maximal solutions initialized in $\{z : \sigma(z) \leq  \overline\alpha^{-1}\circ\underline\alpha(\Delta)\}$ are complete and the bound of Theorem \ref{th:stability-sgpas} holds for all $k \in \Zo$.}

\begin{rem} In Theorem \ref{th:stability-sgpas}, we consider $\theta\in\Theta$ such that $|\theta|\leq\overline\theta$. If there does not exist such a $\theta$, the stability statement in Theorem \ref{th:stability-sgpas} vacuously holds. In any case, the set $\Theta$ is designed by the user so that we can always make sure that such a property holds  but not vacuously.
\end{rem}



\subsection{Stronger stability guarantees}\label{subsect:stronger-stability}

It is possible to derive stronger stability properties compared to Theorem \ref{th:stability-sgpas} by imposing extra conditions on the stopping criterion $c_\stopvi$, which can always be met, and extra assumptions on $\underline\alpha$ and $\overline\alpha$ as in Assumption \ref{ass:linear-bounds}. The corollary below establishes a uniform global asymptotic stability property by only strengthening the condition on $c_\stopvi$ compared to Theorem \ref{th:stability-sgpas}. 

\begin{cor}\label{cor:stability-ugas}  Consider system (\ref{eq:sys-embedding}) with $c_\stopvi$ designed to satisfy Condition \ref{cond:stability} with $\zeta(\sigma(x),|\theta|)\leq|\theta| \alpha(\sigma(x))$ for any $x\in\mathcal{X}$ and $\theta\in\Theta$, with $\alpha$ from SA\ref{sass:optimal-policy-ugas}. {For any $\overline\theta\in(0,\gamma_\star)$ with $\gamma_\star$ from SA\ref{sass:optimal-policy-ugas},} there exists $\beta\in\KL$ such that for any  $\theta\in\Theta$ such that $|\theta|\leq\overline\theta$, for any $x\in\mathcal{X}$, any $i\geq 2$, any $\gamma\in[\gamma_\star,1]$: any solution $\widehat\phi_{\gamma,i,\theta}$ satisfies 
$\sigma(\widehat\phi_{\gamma,i,\theta}(k;x)) \leq \beta(\sigma(x),k)$ 
for all $k\in\dom\widehat\phi_{\gamma,i,\theta}$.
\end{cor}

\begin{proof} Let $\theta\in\Theta$ with $|\theta|<{\overline\theta}$, $x\in\mathcal{X}$, $i\geq 2$,  $\gamma\in[\gamma_\star,1]$ and $\hat{h}_{\gamma,i}\in\widehat{H}_{\gamma,i}$. Like in the proof of Theorem \ref{th:stability-sgpas}, let $Y_{\gamma,i-1}(x):=V_{\gamma,i-1}(f(x,\hat{h}_{\gamma,i}(x)))  - V_{\gamma,i-1}(x)$. Inequality (\ref{eq:proof-stability-lyapunov-bounds}) in the appendix gives under the condition imposed on $\zeta$ in Corollary \ref{cor:stability-ugas},
$Y_{\gamma,i-1}(x)  \leq  -\alpha(\sigma(x))  + \nicefrac{\overline\theta}{\gamma_\star}\alpha(\sigma(x)) =    -(1-\nicefrac{\overline\theta}{\gamma_\star})\alpha(\sigma(x))$. 
As $\overline\theta\in(0,\gamma_\star)$, the desired stability property follows by following the steps of the proof of Theorem \ref{th:stability-sgpas} in Appendix \ref{appendix:proof-main-results}. 
\end{proof}

When Assumption \ref{ass:linear-bounds} holds, a uniform global exponential stability property is established.

\begin{cor}\label{cor:stability-uges} Consider system (\ref{eq:sys-embedding}) and suppose the following holds.
\begin{romanlist}
\item Assumption \ref{ass:linear-bounds} is verified.
\item Condition \ref{cond:stability} holds with $\zeta(\sigma(x),\theta)=|\theta|\sigma(x)$ for any $\theta\in\Theta$ and $x\in\R^{n_x}$. 
\end{romanlist}
Then for any {$\overline\theta\in(0,a\gamma_\star)$ with $a$ as in Lemma \ref{lem:satisfaction-of-sa-gamma-ugas} and $\gamma_\star$ in SA\ref{sass:optimal-policy-ugas}}, any $x\in\mathcal{X}$, any $\theta\in\Theta$ with
$|\theta|\leq \overline \theta$, any $i\geq 2$, any $\gamma\in[\gamma_\star,1]$: any solution $\widehat\phi_{\gamma,i,\theta}$ satisfies 
$\sigma(\widehat \phi_{\gamma,i,\theta}(k;x))  \leq \nicefrac{\overline a}{\underline a}  \lambda ^k \sigma (x)$, 
for any $k\in\dom\widehat\phi_{\gamma,i,\theta}$, with   $\lambda:={1-{(a-\nicefrac{\overline\theta}{\gamma_\star})\overline a^{-1}}}\in(0,1)$. 
\end{cor}

\begin{proof} Let  $\theta\in\Theta$ with $|\theta|\leq\overline\theta $,  $\overline\theta{\in(0, a\gamma_\star)}$,  $i\geq 2$, $\gamma\in[\gamma_\star,1]$ and  $x\in\mathcal{X}$. As before, let $Y_{\gamma,i-1}(x):=V_{\gamma,i-1}(f(x,\hat{h}_{\gamma,i}(x)))  - V_{\gamma,i-1}(x)$. Inequality (\ref{eq:proof-stability-lyapunov-bounds}) in the appendix becomes under the conditions of Corollary \ref{cor:stability-uges} and invoking Lemma \ref{lem:satisfaction-of-sa-gamma-ugas}, 
\begin{eqn}
Y_{\gamma,i-1}(x) \leq {- a \sigma(x) + \nicefrac{\overline\theta}{\gamma_\star}} \sigma(x) 
 =  {-(a-\nicefrac{\overline\theta}{\gamma_\star})}\sigma(x).
\end{eqn}
As $V_{\gamma,i-1}(x)\leq\overline\alpha(\sigma(x)) \leq \overline{a}\sigma(x)$ and $\overline\theta\in(0,a{\gamma_\star})$,
\begin{eqn}
Y_{\gamma,i-1}(x) 
& \leq  &   -{(a-\nicefrac{\overline\theta}{\gamma_\star})}\overline a^{-1} V_{\gamma,i-1}(x) \\
V_{\gamma,i-1}(f(x,\hat{h}_{\gamma,i}(x))) & \leq & {\big(1-(a-\nicefrac{\overline\theta}{\gamma_\star})\overline a^{-1}\big)}V_{\gamma,i-1}(x).
\end{eqn}
We derive that for any solution $\widehat\phi_{\gamma,i,\theta}$ to (\ref{eq:sys-embedding}) initialized at $x$, 
\begin{eqn}
V_{\gamma,i-1}(\widehat \phi_{\gamma,i,\theta}(k;x)) & \leq & {\big(1-(a-\nicefrac{\overline\theta}{\gamma_\star})\overline a^{-1}\big)^k} V_{\gamma,i-1}(x),
\end{eqn}
for any $k\in\dom\widehat \phi_{\gamma,i,\theta}$. As $\underline a\sigma(\widehat \phi_{\gamma,i,\theta}(k))\leq V_{\gamma,i-1}(\widehat \phi_{\gamma,i,\theta}(k))$ and $V_{\gamma,i-1}(x)\leq \overline\alpha(\sigma(x))\leq \overline a \sigma(x)$, we derive that 
\begin{eqn}
\sigma(\widehat \phi_{\gamma,i,\theta}(k;x)) & \leq & \underline a^{-1}{\big(1-(a-\nicefrac{\overline\theta}{\gamma_\star})\overline a^{-1}}\big)^k\overline a \sigma (x),
\end{eqn}
for any $k\in\dom \widehat\phi_{\gamma,i,\theta}$. To complete the proof, we need to show that $\lambda=1-{(a-\nicefrac{\overline\theta}{\gamma_\star})\overline a^{-1}}\in(0,1)$. {This property follows from the facts that $\overline\theta\in(0,a\gamma_\star)$ and $\underline a \leq \overline a$.}
\end{proof}
Corollary \ref{cor:stability-uges}(ii) captures the stopping criterion sometimes considered when applying VI to the LQ case, namely $|P_{i+1}-P_i|\leq \gamma^i\theta$ for some $\theta>0$, like in e.g., \cite{bian-jiang-aut2016(vi)} by taking $\epsilon_i=\gamma^i$, and where the $P_i$'s are the matrices defining the quadratic value function at each iteration, which corresponds to $c_\stopvi(x,\theta,\gamma,i)=\theta\gamma^i\sigma(x)$ with $\sigma(x)=|x|^2$, $\theta\in \Theta\subset\Rlp$.

\begin{rem} Local/semiglobal exponential/asymptotic stability properties as well as global practical stability properties can similarly be derived for system (\ref{eq:sys-embedding}) by imposing appropriate conditions on $\underline\alpha$ and $\overline\alpha$ and following the Lyapunov based analysis in the proof of Theorem \ref{th:stability-sgpas}  like in the proof of Corollary \ref{cor:stability-uges}. These are omitted for space reasons, noting that similar results are derived in  e.g., \cite{Grimm-et-al-tac2005,Postoyan-et-al-tac(optimal),Granzotto-et-al-tac-finite-discounted-horizon}. 
\end{rem}

\section{Near-optimality bounds}\label{sect:near-optimality}

We  derive in this section  near-optimality bounds that explicitly involve $c_\stopvi$. In this way, we can select $c_\stopvi$ to tune the near-optimality bound at the final iteration. We  present two sets of results that rely on different conditions. First, we impose conditions on {$V_{\gamma,\star}$, $V_{\gamma,0}$ (or $V_{\gamma,1}$)} along optimal finite-horizon sequences and solutions to (\ref{eq:sys-optimal}), respectively and derive the desired near-optimality bounds (Section \ref{subsect:properties-V_star-V_0}). These results can always be applied by taking $V_{\gamma,0}\equiv 0$. 
Secondly, we present alternative conditions that rely instead on a summability property of $c_\stopvi$ along the solutions to  (\ref{eq:sys-vi}) and (\ref{eq:sys-optimal}) (Section \ref{subsect:summable-stopping-criterion}). We do not provide properties on the running cost \cite{grune-rantzer-tac08} in this section for space reasons, the interested reader will find such results for the case where $V_{1,0}=0$ and $\gamma=1$ in \cite[Section 4.3]{granzotto-et-al-stop-l4dc}.


\subsection{Near-optimality under conditions on $\mathit{V}_{\gamma,\star}$,  $\mathit{V}_{\gamma,1}$, $\mathit{V}_{\gamma,0}$}\label{subsect:properties-V_star-V_0}


The next theorem provides conditions under which near-optimality bounds involving $c_\stopvi$ are derived. Its proof is in Appendix \ref{appendix:proof-main-results}.

{\begin{thm}\label{th:near-optimality-conditionsV0-Vstar} Suppose there exist functions $\overline{\psi},\underline{\psi}:\Rlo\times(0,1]\times\Zo\to\Rlo$ non-decreasing in their first arguments such that for any iteration $i\in\Zp$, any $\gamma\in[\gamma_\star,1]$ with $\gamma_\star$ from SA\ref{sass:optimal-policy-ugas}, and any $x\in\mathcal{X}$,
\begin{equation}
\begin{array}{llll}\label{eq:sass-V0}
V_{\gamma,\star}(\varphi_{\gamma,i}) - V_{\gamma,1}(\varphi_{\gamma,i}) \\
\hspace{2cm}\leq  \overline{\psi}(\nicefrac{1}{\gamma^i}|V_{\gamma,i+1}(x)-V_{\gamma,i}(x)|,\gamma,i) \\
V_{\gamma,0}(\phi_{\gamma,\star}(i+1;x)) -V_{\gamma,\star}(\phi_{\gamma,\star}(i+1;x)) \\
\hspace{2cm}\leq   \underline{\psi}(\nicefrac{1}{\gamma^i}|V_{\gamma,i+1}(x)-V_{\gamma,i}(x)|,\gamma,i),
\end{array}
\end{equation}
for any sequence $\varphi_{\gamma,k}$, $k\in\{0,\ldots,i+1\}$, with $\varphi_{\gamma,0}=x$ and $\varphi_{\gamma,k+1}\in F_{\gamma,i+1-k}(\varphi_{\gamma,k})$ for $k\in\{0,\ldots,i\}$, and any solution $\phi_{\gamma,\star}$ to (\ref{eq:sys-optimal}).
Then for any $\theta\in\Theta$, $\gamma\in[\gamma_\star,1]$, any iteration $i\in\Zp$ such that (\ref{eq:stopping-criterion}) holds,  any $x\in\mathcal{X}$,
\begin{equation}
\begin{array}{llll}\label{eq:near-optimality-bound-under-extra-ass}
|V_{\gamma,i+1}(x) - V_{\gamma,\star}(x)| \leq  \gamma^i \max \big\{\overline{\psi}\big(\gamma^{-i} c_{\stopvi}(x,\theta,\gamma,i), \gamma, i\big),\\
\hspace{4.65cm}\underline{\psi}\big(\gamma^{-i} c_{\stopvi}(x,\theta,\gamma,i), \gamma, i\big) \big\}.
\end{array}
\end{equation}
\end{thm}}

The upper-bound in property (\ref{eq:near-optimality-bound-under-extra-ass}) explicitly depends on the tunable stopping criterion $c_\stopvi$. It also depends on the number of iterations, which implicitly depends on the stopping criterion. This may seem problematic as, typically the smaller $c_\stopvi(x,\theta,\gamma,i)$, the larger $i$. However, the bound is still computable and informative. Also, in the special case where $\underline\psi$ and $\overline\psi$ are independent of $\gamma$ and $i$ {(like in Proposition \ref{prop:conditions-sa-V-0} below)} so that we can omit the corresponding arguments and of non-decreasing, we can always design $c_\stopvi$ to be such that $c_\stopvi(x,\theta,\gamma,i)=\gamma^{i}\bar{c}(\theta,x)$ with $\bar c:\Theta\times\R^{n_x}\to\Rlo$, and in this way the inequality (\ref{eq:near-optimality-bound-under-extra-ass}) becomes
\begin{eqn}
|V_{\gamma,{i+1}}(x) -  V_{\gamma,\star}(x)| &  \leq &\gamma^i\max\{\overline{\psi},\underline{\psi}\}(\bar c(\theta,x)) \\
& \leq & \max\{\overline{\psi},\underline{\psi}\}(\bar c(\theta,x)).
\end{eqn}
We can then select  $\bar c$ to attain the desired near-optimality bound, with guaranteed  termination by  Theorems \ref{th:finite-iteration-condition-varepsilon-gamma} and \ref{th:finite-iteration-condition-L-sigma(x)-gamma} under their conditions. Another relevant special case is  when  there exists $L_\psi\in\Rlp$ such that $\max\{\overline\psi,\underline\psi\}(s,\gamma,i)\leq L_\psi s$ for any $(s,\gamma,i)\in\Rlo\times(0,1]\times\Zo$, then the upper-bound in (\ref{eq:near-optimality-bound-under-extra-ass}) becomes $L_\psi c_\stopvi(x,\theta,\gamma,i)$, which can be arbitrarily tuned using $c_\stopvi$.  



Contrary to the near-optimality bounds found in the literature, like in \cite{bertsekas2012dynamic}, the near-optimality bound in (\ref{eq:near-optimality-bound-under-extra-ass}) is finite when $\gamma=1$, and explicitly depends on the function used to define the stopping criterion as opposed to \cite{Granzotto-et-al-tac-finite-discounted-horizon}.  

The near-optimality property in (\ref{eq:near-optimality-bound-under-extra-ass}) applies  under several conditions. First, we need the stopping criterion to be satisfied; the results of Section \ref{sect:finite-iterations}  address this question.  
Second,  the property in (\ref{eq:sass-V0}) has to hold. At a first glance, this property may seem difficult to test as it involves $V_{\gamma,i}$, $V_{\gamma,i+1}$ and the sequence $\varphi_{\gamma,i}$, that we can obviously not compute a priori, as well as  $V_{\gamma,\star}$ and $\phi_{\gamma,\star}$ that we do not know in general. The next proposition provides sufficient conditions that overcome this issue. Its proof is in Appendix \ref{appendix:proof-main-results}.

\begin{prop}\label{prop:conditions-sa-V-0} Suppose there exists $\alpha_0\in\Kinf$ such that the following holds for any $x\in\R^{n_x}$ and $\gamma\in[\gamma_\star,1]$ with $\gamma_\star$ from SA\ref{sass:optimal-policy-ugas}. 
\begin{enumerate}[label=(\roman*)]
\item  $V_{\gamma,0}(x)\leq V_{\gamma,\star}(x)$.
\item $\min_{u\in\mathcal{U}(x)}\!\!\left(\gamma V_{\gamma,0}(f(x,u))-V_{\gamma,0}(x)+\ell(x,u)\right)\geq  \alpha_0(\sigma(x))$.
\end{enumerate}
Then  property (\ref{eq:sass-V0}) is satisfied with $\overline{\psi}(s,\gamma,i)=\overline\alpha\circ\alpha_0^{-1}(s)$ for any $(s,\gamma,i)\in\Rlo\times(0,1]\times\Zo$ and  $\underline{\psi}=0$. 
\end{prop}

The conditions of Proposition \ref{prop:conditions-sa-V-0} are easier to test than those of Theorem \ref{th:near-optimality-conditionsV0-Vstar}. For instance, these hold when $V_{\gamma,0}(x)=0$ for any $x\in\R^{n_x}$, given SA\ref{sass:detectability}. More generally, the requirements of Proposition \ref{prop:conditions-sa-V-0} are verified when $V_{\gamma,0}(x)\in[0,\varsigma \underline\alpha(\sigma(x))]$ for any $x\in\R^{n_x}$ for some $\varsigma\in[0,1)$ and where $\underline\alpha$ comes from SA\ref{sass:detectability} as formalized next.

\begin{lem}\label{lem:item(ii)-of-proposition} Given any function $V_{\gamma,0}:\R^{n_x}\to\Rlo$ for which  there exists $\varsigma\in[0,1)$ such that  $V_{\gamma,0}(x)\leq \varsigma\underline\alpha(\sigma(x))$ for any $x\in\R^{n_x}$ with $\underline\alpha$ in SA\ref{sass:detectability}, the conditions of  Proposition \ref{prop:conditions-sa-V-0} hold with $\alpha_0=(1-\varsigma)\underline\alpha$. 
\end{lem}

\begin{proof} By SA\ref{sass:detectability} and the definition of $V_{\gamma,\star}$ in (\ref{eq:optimal-value-function}), for any $x\in\R^{n_x}$, $\underline\alpha(\sigma(x))\leq V_{\gamma,\star}(x)$. Consequently, as $V_{\gamma,0}(x)\leq \varsigma\underline\alpha(\sigma(x))$ for any $x\in\R^{n_x}$ with $\varsigma<1$, $V_{\gamma,0}\leq V_{\gamma,\star}$, i.e., Proposition \ref{prop:conditions-sa-V-0}(i) holds. Now, let $x\in\R^{n_x}$ and $\gamma\in[\gamma_\star,1]$, by SA\ref{sass:detectability},
$\min_{u\in\mathcal{U}(x)} \left( \gamma V_{\gamma,0}(f(x,u)) - V_{\gamma,0}(x) + \ell(x,u) \right) 
    \geq -V_{\gamma,0}(x) + \underline{\alpha}(\sigma(x))$.
As $V_{\gamma,0}(x)\leq \varsigma\underline\alpha(\sigma(x))$ by assumption,
$ \min_{u\in\mathcal{U}(x)} \big(\gamma V_{\gamma,0}(f(x,u)) - V_{\gamma,0}(x) + \ell(x,u) \big) 
    \geq -\varsigma \underline{\alpha}(\sigma(x)) + \underline{\alpha}(\sigma(x))
    = (1-\varsigma) \underline{\alpha}(\sigma(x))$. 
from which we derive the desired result by taking $\alpha_0=(1-\varsigma)\underline\alpha\in\Kinf$ as $\varsigma\in[0,1)$.
\end{proof}


\subsection{Summable stopping criterion along closed-loop system solutions}\label{subsect:summable-stopping-criterion}

Throughout this section, to streamline the developments we assume 
{that $\mathcal{X}=\R^{n_x}$ in (\ref{eq:stopping-criterion}) and that system (\ref{eq:sys-vi}) satisfies a global attractivity  property} {at the final iteration}. {This} is the case, {for instance}, under the conditions of Corollary \ref{cor:stability-ugas}. The next theorem provides near-optimality bounds on the mismatch between the value function satisfying (\ref{eq:stopping-criterion}) and the optimal value function $V_{\gamma,\star}$ in (\ref{eq:optimal-value-function}). The proof is given in Appendix \ref{appendix:proof-main-results}.

\begin{thm}\label{th:near-optimality-bounds-summable} Suppose {$\mathcal{X}=\R^{n_x}$ and let $\theta\in\Theta$, $\gamma\in[\gamma_\star,1]$ and $i\in\Zp$ be
such that (\ref{eq:stopping-criterion}) holds and $\sigma(\phi_{\gamma,i+1}(k;x))\to 0$ as $k\to\infty$ for
any $x\in\R^{n_x}$ and any solution $\phi_{\gamma,i+1}$ to (\ref{eq:sys-vi})}.
Then for any $x\in\R^{n_x}$,
$|V_{\gamma,i+1}(x) - V_{\gamma,\star}(x)| \leq  \max (S_{i+1}(x,\gamma,i),S_\star(x,\gamma,i))$
where $S_{i+1}(x,\gamma,i):= \inf_{\phi_{\gamma,i+1}}\sum_{j={1}}^\infty \gamma^j c_{\stopvi}(\phi_{\gamma,i+1}(j;x), \theta,\gamma, i),$ and $S_\star(x,\gamma,i):=  \inf_{\phi_{\gamma,\star}} \sum_{j={1}}^\infty \gamma^j c_{\stopvi}(\phi_{\gamma,\star}(j;x),\theta, \gamma, i)$ where we recall $\phi_{\gamma,i+1}$ and $\phi_{\gamma,\star}$ are  solutions  to systems (\ref{eq:sys-vi}) and (\ref{eq:sys-optimal}), respectively.
\end{thm}

The upper-bound of the inequality in Theorem \ref{th:near-optimality-bounds-summable} is finite in the next cases to give few examples. When  $\gamma<1$ and $\left(c_\stopvi(\phi_{\gamma,i+1}(k;x),\theta,\gamma,i)\right)_{k\in\Zp}$ and $\left(c_\stopvi(\phi_{\gamma,\star}(k;x),\theta,\gamma,i)\right)_{k\in\Zp}$ are bounded sequences, for any solutions $\phi_{\gamma,i+1}$ and $\phi_{\gamma,\star}$ to (\ref{eq:sys-vi}) at $i+1$ such that (\ref{eq:stopping-criterion}) holds and (\ref{eq:sys-optimal}), respectively. When $\gamma=1$, we need the extra property that $\left(c_\stopvi(\phi_{\gamma,i+1}(k;x),\theta,1,i)\right)_{k\in\Zp}$ and $\left(c_\stopvi(\phi_{\gamma,\star}(k;x),\theta,1,i)\right)_{k\in\Zp}$ are summable, which cannot hold under Condition \ref{cond:finite-iteration}(i) but may hold under Condition \ref{cond:finite-iteration}(ii) and Condition \ref{cond:stability}. 
The next {proposition} presents sufficient conditions under which this is the case.

{\begin{prop}\label{prop:conditions-summability} Suppose Assumption \ref{ass:linear-bounds} holds, $\mathcal{X}=\R^{n_x}$, and $c_\stopvi$ satisfies Corollary \ref{cor:stability-uges}(ii). Let $\overline\theta\in(0,a\gamma_\star)$ with $a=\underline a-\overline a+\gamma_\star \overline a$ as in Lemma \ref{lem:satisfaction-of-sa-gamma-ugas} and $\gamma_\star$ from SA\ref{sass:optimal-policy-ugas}, and $\lambda=1-(a-\nicefrac{\overline\theta}{\gamma_\star})\overline a^{-1}\in(0,1)$ as in Corollary \ref{cor:stability-uges}. Then, for any $\theta\in\Theta$ with $|\theta|\leq\overline\theta$, any $\gamma\in[\gamma_\star,1]$ and any $i\in\Zp$ such that (\ref{eq:stopping-criterion}) holds,
\begin{eqn} \label{eq:lem-near-optimality-summable}
    \max \big(S_{i+1}(x,\gamma,i), S_{\star}(x,\gamma,i) \big)
    \leq  \frac{\overline{a}|\theta|\gamma\lambda}{\underline{a}(1-\gamma\lambda)} \sigma(x)
\end{eqn}
for any $x\in\R^{n_x}$, with $S_{i+1}$ and $S_\star$ as in Theorem~\ref{th:near-optimality-bounds-summable}. Consequently,
\begin{eqn}\label{eq:prop-near-optimality-clean-bound}
|V_{\gamma,i+1}(x)-V_{\gamma,\star}(x)| & \leq & \frac{\overline a|\theta|\gamma\lambda}{\underline a(1-\gamma\lambda)}\sigma(x)
\end{eqn}
for any $x\in \R^{n_x}$.
\end{prop}}

\begin{proof} Let $\gamma\in[\gamma_\star,1]$,  $x\in\R^{n_x}$ and $\phi_{\gamma,i+1}$ and $\phi_{\gamma,\star}$ be solutions to (\ref{eq:sys-vi}) at $i+1$ such that (\ref{eq:stopping-criterion}) holds with $\theta\in\Theta$  and (\ref{eq:sys-optimal}), respectively. Take $j\in\Zp$, by application of Corollary \ref{cor:stability-uges},
\begin{eqn}
c_{\stopvi}(\phi_{\gamma,i+1}(j;x),\theta,\gamma,i) & \leq & |\theta|\sigma(\phi_{\gamma,i+1}(j;x)) \\
& \leq & \nicefrac{\overline a}{\underline a}|\theta|\lambda^j \sigma(x).
\end{eqn}
Consequently, noting that $\gamma\in(0,1]$ as $\gamma\in[\gamma_\star,1]$ and $\lambda\in(0,1)$ as shown in the proof of Corollary \ref{cor:stability-uges}, 
\begin{eqn}
\sum_{j=1}^{k}\gamma^j c_{\stopvi}(\phi_{\gamma,i+1}(j;x),\theta,\gamma,i) 
& \leq & \nicefrac{\overline a}{\underline a}|\theta| \frac{{\gamma \lambda}}{1-\gamma\lambda}\sigma(x).
\end{eqn}
By taking the limit as $k\to\infty$, which exists as the above series is non-decreasing with $k$, 
\begin{eqn}\label{eq:proof-summability-phi-i+1}
{\lim\limits_{k\to\infty}}\sum_{j=1}^{k}\gamma^j c_{\stopvi}(\phi_{\gamma,i+1}(j;x),\theta,\gamma,i) 
& \leq & \nicefrac{\overline a }{\underline a}|\theta| \frac{{\gamma\lambda}}{1-\gamma\lambda}\sigma(x).
\end{eqn}
By following similar lines as in the proof of Corollary \ref{cor:stability-uges}, we derive that, noting that ${1-a\overline a^{-1}\leq \lambda}$, 
\begin{eqn}
\sigma(\phi_{\gamma,\star}(j;x)) 
& \leq & {\nicefrac{\overline a}{\underline a}(1-a \overline a^{-1})^j} \sigma(x) \leq \nicefrac{\overline a}{\underline a}\lambda^j\sigma(x).
\end{eqn}
Using the same reasoning as above, we obtain that
\begin{eqn}\label{eq:proof-summability-phi-star}
\lim_{k\to\infty}\sum_{j=1}^{k}\gamma^j c_{\stopvi}(\phi_{\gamma,\star}(j;x),\theta,\gamma,i) 
& \leq & \nicefrac{\overline a}{\underline a}|\theta| \frac{{\gamma\lambda}}{1-\gamma\lambda}\sigma(x).
\end{eqn}
We then derive  (\ref{eq:lem-near-optimality-summable}). The last statement of Proposition \ref{prop:conditions-summability} follows by application of Theorem \ref{th:near-optimality-bounds-summable}{, noting that the required  conditions for this purpose hold.}
\end{proof}

Proposition \ref{prop:conditions-summability} establishes that the conditions of Theorem \ref{th:near-optimality-bounds-summable} hold under {some of} the conditions of Corollary \ref{cor:stability-uges}. As a result, easily computable bounds on the mismatch $|V_{\gamma,i+1}-V_{\gamma,\star}|$ are derived in (\ref{eq:prop-near-optimality-clean-bound}).

\section{Example}\label{sect:example}

Consider the discrete cubic integrator \cite[Example 1]{Grimm-et-al-tac2005}, {whose origin cannot be stabilized by a continuous state feedback \cite{meadows1995receding}.} For this example, $n_x=2$, $n_u=1$, and $f(x,u)=(x_1+u,x_2+u^3)$ for any $x=(x_1,x_2)\in\R^2$ and $u\in\R$. The stage cost is $\ell(x,u)=\sigma(x)+|u|^3$  with $\sigma(x)=|x_1|^3+|x_2|$ for any $(x,u)\in\R^{3}$ and the initial value function is taken as $V_{\gamma,0}(x)=|x_2|$ for any $x\in\R^2$. SA\ref{sass:well-posedness} and SA\ref{sass:detectability} hold with $\underline\alpha(s) = s$ for any $s\geq 0$ for the latter. SA\ref{sass:stabilizability} holds with $\overline \alpha=14\id$ by \cite[Example 1]{Granzotto-et-al-cdc2018} noting that  $V_{\gamma,0}(x)\leq \sigma(x)$ for any $x\in\R^2$. Hence Assumption \ref{ass:linear-bounds} is verified with $\underline a=1$ and $\overline a=14$. By Lemma \ref{lem:satisfaction-of-sa-gamma-ugas},  we derive that SA\ref{sass:optimal-policy-ugas} is satisfied with $\gamma_\star\in(\nicefrac{13}{14},1]$.  On the other hand, the conditions of Proposition \ref{prop:conditions-sa-V-0} hold. Indeed, let $x\in\R^2$ and $\gamma\in[\gamma_\star,1]$,  $V_{\gamma,0}(x)\leq\ell(x,0)\leq V_{\gamma,\star}(x)$. {Moreover, for any $u\in\R$, $\gamma V_{\gamma,0}(f(x,u))-V_{\gamma,0}(x)+\ell(x,u)=|x_1|^3+\gamma|x_2+u^3|+|u|^3\geq|x_1|^3+\gamma|x_2|+(1-\gamma)|u|^3\geq\gamma_\star\sigma(x)$, so that Proposition \ref{prop:conditions-sa-V-0}(ii) holds with $\alpha_0=\gamma_\star\id$.} 


\begin{table*}[t]
\centering
\renewcommand{\arraystretch}{1.2}
\begin{tabular}{llllll}
Function $c_\stopvi$ 
& \multicolumn{3}{l}{Final iteration} 
& Stability guarantee 
& Bound on $|V_{\gamma,i_\stopvi}(x) - V_{\gamma,\star}(x)|$ \\
& ($\theta=0.1)$ & $(\theta=0.01)$ & $(\theta=0.001)$ & for system (\ref{eq:sys-embedding}) & (Prop. \ref{prop:conditions-sa-V-0} and Thm. \ref{th:near-optimality-conditionsV0-Vstar})\\
\midrule
$\theta$ & {$13$} & {$15$} & {$18$} & Semiglobal practical (Thm. \ref{th:stability-sgpas})
& $\nicefrac{14}{\gamma_\star}\theta$ \\
$\gamma^{i}\theta$ & {$13$} & {$16$} & {$18$} & Semiglobal practical (Thm. \ref{th:stability-sgpas})
& $\nicefrac{14\gamma^{i_\stopvi-1}}{\gamma_\star}\theta$ \\
$\theta\sigma(x)$ & {$7$} & {$9$} & {$12$} & Unif. global exponential (Cor. \ref{cor:stability-uges})
& $\nicefrac{14}{\gamma_\star}\theta\sigma(x)$ \\
$\gamma^{i}\theta\sigma(x)$ & {$7$} & {$9$} & {$12$} & Unif. global exponential (Cor. \ref{cor:stability-uges})
& $\nicefrac{14 \gamma^{i_\stopvi-1}}{\gamma_\star}\theta\sigma(x)$ \\
\midrule
\end{tabular}
\caption{Stopping criteria for the example in Section~\ref{sect:example} and their properties, {for $\gamma=0.98$. The final iterations are computed with $\mathcal{X}=\{x\in\R^2\,:\,\sigma(x)\leq 200\}$ for the first two criteria and $\mathcal{X}=\R^2$ for the last two}.}
\label{tab:example}
\end{table*}

{We consider $\gamma=0.98$ and $\gamma_\star=0.97$, so that $a=0.58$ in Lemma~\ref{lem:satisfaction-of-sa-gamma-ugas} and $\theta<a\gamma_\star$ in Corollary~\ref{cor:stability-uges} for all considered values of $\theta$. Moreover, $\gamma>1-\overline a^{-1}a\approx0.959$ as required in Theorems~\ref{th:finite-iteration-condition-varepsilon-gamma} and~\ref{th:finite-iteration-condition-L-sigma(x)-gamma}, a condition that holds for any $\gamma_\star>0.95$ since $1-\overline a^{-1}a=2-\nicefrac{\underline a}{\overline a}-\gamma_\star$ by definition of $a$. We compare four stopping criteria as in Table~\ref{tab:example} with different values of $\theta\in\Theta=\Rlp$. The corresponding guaranteed stability properties for system (\ref{eq:sys-vi-stopping}) at the final iteration as well as the near-optimality bounds are reported in this table.
To compute the final iterations, we implement VI by exploiting the homogeneity of the example, inspired by \cite{granzottoExploitingHomogeneityOptimal2021,rinehart-et-al-tac09(vi)}. Indeed, given $\mathbf{d}_\lambda(x):=(\lambda x_1,\lambda^3x_2)$ with $\lambda>0$, we have $f(\mathbf{d}_\lambda(x),\lambda u)=\mathbf{d}_\lambda(f(x,u))$, and $\ell$, $\sigma$ and $V_{\gamma,0}$ are scaled by $\lambda^3$, so that (\ref{eq:vi-as-finite-horizon}) gives $V_{\gamma,i}(\mathbf{d}_\lambda(x))=\lambda^3V_{\gamma,i}(x)$ for any $x\in\R^2$ and $i\in\Zo$. Hence VI only needs to be computed on the curve $\{x\in\R^2\,:\,\sigma(x)=1\}$, which we discretize with $N=5001$ nodes and linear interpolation, using $M=1001$ input values uniformly sampled in [-2.5, 2.5]. The results are essentially unchanged for finer resolutions. The stopping criteria $\theta\sigma(x)$ and $\gamma^i\theta\sigma(x)$ are then checked on $\mathcal{X}=\R^2$, and the criteria $\theta$ and $\gamma^i\theta$ on $\mathcal{X}=\{x\in\R^2\,:\,\sigma(x)\leq 200\}$. As expected, decreasing $\theta$ requires additional iterations. Interestingly, the criteria involving $\sigma(x)$ can be checked on the whole state space and, for the
considered set $\mathcal{X}$ for the first two criteria, require fewer iterations while providing stronger
stability and near-optimality guarantees; the factor $\gamma^{i}$ costs at most one additional iteration.}


\section{Conclusion}\label{sect:conclusion}

We considered a general class of stopping criteria for VI, that encompasses those commonly encountered in dynamic programming and reinforcement learning as well as new ones, and {presented conditions under which there exists a finite final iteration at which the stopping criterion holds}. Afterwards, we have derived conditions under which any final policy exhibits stabilizing properties. We have also developed novel near-optimality bounds involving the selected stopping criteria. 

Two promising avenues for future research emerge from this work. First, it would be valuable to explicitly incorporate the approximation errors that may arise during the iterative computation of the value function. Second,  these findings could be extended to stochastic discrete-time systems by exploiting the  results in \cite{moldenhauer2026-tac}.

\section{Appendix}

\subsection{Proofs of the main results}\label{appendix:proof-main-results}

\noindent\textbf{Proof of Theorem \ref{th:stability-sgpas}.} Let $\delta,\Delta>0$ {with $\delta\leq\Delta$ without loss of generality,}  $\gamma\in[\gamma_\star,1]$, $i\geq 2$, $\theta\in\Theta$ with $|\theta|\leq\overline\theta $ with $\overline\theta>0$ determined in the following, $x\in\mathcal{X}$ with $\sigma(x)\leq\Delta$ and $\hat{h}_{\gamma,i}\in\widehat H_{\gamma,i}(\cdot,\theta)$. We also introduce $\widetilde\Delta:=\overline\alpha(\Delta)$ as well as $\tilde\delta:={\underline\alpha(\delta)}$ {and} $\widetilde\alpha:=\alpha\circ\overline\alpha^{-1}$ {with}  $\id-\widetilde\alpha\in\Kinf$ without loss of generality by \cite[Remark 2]{Granzotto-et-al-tac-finite-discounted-horizon}. Let $Y_{\gamma,i-1}(x):=V_{\gamma,i-1}(f(x,\hat{h}_{\gamma,i}(x)))  - V_{\gamma,i-1}(x)$, by Theorem \ref{th:lyapunov}(ii),
\begin{eqn}
Y_{\gamma,i-1}(x) & \leq & -\alpha(\sigma(x))  + \frac{1}{\gamma_\star}c_\stopvi(x,\theta,\gamma,{i-1}),
\end{eqn}
and by Condition \ref{cond:stability},
\begin{eqn}\label{eq:proof-stability-lyapunov-bounds}
Y_{\gamma,i-1}(x) & \leq & -\alpha(\sigma(x))  + \frac{1}{\gamma_\star}\zeta(\sigma(x),|\theta|).
\end{eqn}
As $|\theta|\leq\overline\theta$ and $\zeta$ is increasing in its second argument,
\begin{eqn}
Y_{\gamma,i-1}(x) & \leq & -\alpha(\sigma(x)) + \frac{1}{\gamma_\star}\zeta(\sigma(x),\overline\theta).
\end{eqn}
By Theorem \ref{th:lyapunov}(i), $\alpha\circ\overline\alpha^{-1}(V_{\gamma,i-1}(x))\leq\alpha(\sigma(x))$ and $\zeta(\sigma(x),\overline\theta)\leq\zeta(\underline\alpha^{-1}(V_{\gamma,i-1}(x)),\overline\theta)$ as $\zeta$ in Condition \ref{cond:stability} is non-decreasing in its first argument. As a result,
\begin{eqn}
Y_{\gamma,i-1}(x) & \leq & -\alpha\circ\overline\alpha^{-1}(V_{\gamma,i-1}(x)) \\
& & + \frac{1}{\gamma_\star}\zeta(\underline\alpha^{-1}(V_{\gamma,i-1}(x)),\overline\theta).
\end{eqn}
For the sake of convenience, we introduce  $\widetilde\zeta:=\frac{1}{\gamma_\star}\zeta(\underline\alpha^{-1}(\cdot),\cdot)$ so that the last inequality becomes in view of the definition of $\widetilde\alpha$ above
\begin{eqn}
Y_{\gamma,i-1}(x) & \leq & -\widetilde\alpha(V_{\gamma,i-1}(x)) + \widetilde\zeta(V_{\gamma,i-1}(x),\overline\theta).
\end{eqn}
Select $\overline\theta$ sufficiently small such that 
\begin{eqn}\label{eq:proof-cond-zeta}
\widetilde{\zeta}(s,\overline\theta) \leq \frac{1}{2}\widetilde\alpha(s) & & \forall s\in[\tilde\delta,\widetilde\Delta],
\end{eqn}
which is always possible given the properties of $\zeta$ in Condition \ref{cond:stability}. Indeed, it suffices to select $\overline\theta$ such that $\widetilde{\zeta}(\widetilde\Delta,\overline\theta)\leq \frac{1}{2}\widetilde\alpha(\tilde\delta)$ for (\ref{eq:proof-cond-zeta}) to hold. Note that, as $\sigma(x)\leq \Delta$ and $V_{\gamma,i-1}(x)\leq\overline\alpha(\sigma(x))$ by SA\ref{sass:stabilizability}, $V_{\gamma,i-1}(x)\leq \widetilde\Delta$ by definition of $\widetilde\Delta$. When $V_{\gamma,i-1}(x)\in[\widetilde\delta,\widetilde\Delta]$, by (\ref{eq:proof-cond-zeta}),
\begin{eqn}\label{eq:proof-lyap-bigger-tilde-delta}
Y_{\gamma,i-1}(x) & \leq & -\frac{1}{2}\widetilde\alpha(V_{\gamma,i-1}(x))-\frac{1}{2}\widetilde\alpha(V_{\gamma,i-1}(x)) \\
& & + \widetilde\zeta(V_{\gamma,i-1}(x),\overline\theta) \\
& \leq & -\frac{1}{2}\widetilde\alpha(V_{\gamma,i-1}(x)).
\end{eqn}
On the other hand, when $V_{\gamma,i-1}(x)\leq\widetilde\delta$, as $\widetilde\zeta$ is non-decreasing in its first argument, 
\begin{eqn}\label{eq:proof-lyap-smaller-tilde-delta-pre}
V_{\gamma,i-1}(f(x,\hat{h}_{\gamma,i}(x))) & \leq &  V_{\gamma,i-1}(x) -\widetilde\alpha(V_{\gamma,i-1}(x)) \\
& & + \widetilde\zeta(V_{\gamma,i-1}(x),\overline\theta)\\
& \leq & V_{\gamma,i-1}(x) -\widetilde\alpha(V_{\gamma,i-1}(x)) + \widetilde\zeta(\widetilde\delta,\overline\theta)
\end{eqn}
and as $\id-\widetilde\alpha\in\Kinf$ without loss of generality (see again \cite[Remark 2]{Granzotto-et-al-tac-finite-discounted-horizon}), 
\begin{eqn}
V_{\gamma,i-1}(f(x,\hat{h}_{\gamma,i}(x))) 
& \leq & \widetilde\delta -\widetilde\alpha(\widetilde\delta) + \widetilde\zeta(\widetilde\delta,\overline\theta).
\end{eqn}
By (\ref{eq:proof-cond-zeta}), $\widetilde\zeta(\widetilde\delta,\overline\theta)\leq\frac{1}{2}\widetilde\alpha(\widetilde\delta)$, therefore
\begin{eqn}\label{eq:proof-lyap-smaller-tilde-delta}
V_{\gamma,i-1}(f(x,\hat{h}_{\gamma,i}(x))) 
& \leq & \widetilde\delta -\widetilde\alpha(\widetilde\delta) + \frac{1}{2}\widetilde\alpha(\widetilde\delta)  \leq \widetilde\delta.
\end{eqn}
Given (\ref{eq:proof-lyap-bigger-tilde-delta}) and (\ref{eq:proof-lyap-smaller-tilde-delta}), we follow similar steps as in the proof of \cite[Thm. 2]{Postoyan-et-al-tac(optimal)} to derive the existence of $\widetilde\beta\in\KL$ independent of $\delta,\Delta,i,\theta$ such that for any solution $\widehat\phi_{\gamma,i,\theta}$ to (\ref{eq:sys-embedding}), for any $k\in\dom \widehat\phi_{\gamma,i,\theta}$,
\begin{eqn}
V_{\gamma,i-1}(\widehat\phi_{\gamma,i,\theta}(k;x)) & \leq & \max\{\widetilde\beta(\sigma(x),k),\widetilde\delta\}.
\end{eqn}
We derive the desired result using Theorem \ref{th:lyapunov}(i). \hfill $\Box$\\

\noindent\textbf{Proof of Theorem \ref{th:near-optimality-conditionsV0-Vstar}.} Let $\theta\in\Theta$, $\gamma\in[\gamma_\star,1]$ and $i\in\Zp$ be such that $|V_{\gamma,i+1}(z)-V_{\gamma,i}(z)|\leq c_{\stopvi}(z,\theta,\gamma,i)$ for any $z\in\mathcal{X}$, and let $x\in\mathcal{X}$. Consider a sequence $\varphi_{\gamma,k}$, $k\in\{0,\ldots,i+1\}$, with $\varphi_{\gamma,0}=x$ and $\varphi_{\gamma,k+1}=f(\varphi_{\gamma,k},u_k)$ with $u_k\in H_{\gamma,i+1-k}(\varphi_{\gamma,k})$ for $k\in\{0,\ldots,i\}$, which exists by Theorem \ref{th:recursive-feasibility}. By (\ref{eq:vi-value})-(\ref{eq:vi-policy}), $V_{\gamma,i+1-k}(\varphi_{\gamma,k})=\ell(\varphi_{\gamma,k},u_k)+\gamma V_{\gamma,i-k}(\varphi_{\gamma,k+1})$ for any $k\in\{0,\ldots,i-1\}$, hence, by iterating,
\begin{eqn}\label{eq:proof-near-opt-dp}
V_{\gamma,i+1}(x) & = & \sum_{k=0}^{i-1}\gamma^k \ell(\varphi_{\gamma,k},u_k)+\gamma^i V_{\gamma,1}(\varphi_{\gamma,i}).
\end{eqn}
On the other hand, by definition of $V_{\gamma,\star}$ in (\ref{eq:optimal-value-function}) and Lemma \ref{lem:existence-optimal-inputs}, concatenating $(u_0,\ldots,u_{i-1})$ with an optimal input sequence from $\varphi_{\gamma,i}$ gives
\begin{eqn}
V_{\gamma,\star}(x) & \leq & \sum_{k=0}^{i-1}\gamma^k \ell(\varphi_{\gamma,k},u_k)+\gamma^i V_{\gamma,\star}(\varphi_{\gamma,i}).
\end{eqn}
Combining the last two relations and using (\ref{eq:sass-V0}), (\ref{eq:stopping-criterion}) at $x$ and the fact that $\overline{\psi}$ is non-decreasing in its first argument,
\begin{eqn}\label{eq:proof-near-optimality-bound-1}
V_{\gamma,\star}(x) & \leq & V_{\gamma,i+1}(x) + \gamma^i\big(V_{\gamma,\star}(\varphi_{\gamma,i})-V_{\gamma,1}(\varphi_{\gamma,i})\big)\\
& \leq & V_{\gamma,i+1}(x) + \gamma^i\overline{\psi}(\nicefrac{1}{\gamma^i}c_\stopvi(x,\theta,\gamma,i),\gamma,i).
\end{eqn}

On the other hand, let $\phi_{\gamma,\star}$ be a solution to (\ref{eq:sys-optimal}) obtained with an optimal policy $h_{\gamma,\star}\in H_{\gamma,\star}$. By (\ref{eq:vi-as-finite-horizon}) with $i+1$ in place of $i$,
\begin{eqn}
V_{\gamma,i+1}(x) & \leq & \sum_{k=0}^{i}\gamma^k\ell(\phi_{\gamma,\star}(k;x),h_{\gamma,\star}(\phi_{\gamma,\star}(k;x)))\\
& & +\gamma^{i+1} V_{\gamma,0}(\phi_{\gamma,\star}(i+1;x)),
\end{eqn}
while, by iterating the Bellman equation $V_{\gamma,\star}(z)=\ell(z,h_{\gamma,\star}(z))+\gamma V_{\gamma,\star}(f(z,h_{\gamma,\star}(z)))$,
\begin{eqn}
V_{\gamma,\star}(x) & = & \sum_{k=0}^{i}\gamma^k\ell(\phi_{\gamma,\star}(k;x),h_{\gamma,\star}(\phi_{\gamma,\star}(k;x)))\\
& & +\gamma^{i+1} V_{\gamma,\star}(\phi_{\gamma,\star}(i+1;x)).
\end{eqn}
Hence, by (\ref{eq:sass-V0}), (\ref{eq:stopping-criterion}) at $x$, the facts that $\underline{\psi}$ is non-decreasing in its first argument, takes non-negative values and $\gamma^{i+1}\leq\gamma^i$,
\begin{eqn}\label{eq:proof-near-optimality-bound-2}
V_{\gamma,i+1}(x) & \leq &  V_{\gamma,\star}(x) + \gamma^{i+1}\big(V_{\gamma,0}(\phi_{\gamma,\star}(i+1;x))\\
& & -V_{\gamma,\star}(\phi_{\gamma,\star}(i+1;x))\big)\\
& \leq & V_{\gamma,\star}(x) + \gamma^i \underline{\psi}(\nicefrac{1}{\gamma^i} c_\stopvi(x,\theta,\gamma,i),\gamma,i).
\end{eqn}
Inequality (\ref{eq:near-optimality-bound-under-extra-ass}) follows from (\ref{eq:proof-near-optimality-bound-1}) and (\ref{eq:proof-near-optimality-bound-2}). \hfill $\Box$

\noindent\textbf{Proof of Proposition \ref{prop:conditions-sa-V-0}.} Proposition \ref{prop:conditions-sa-V-0}(i) implies the second inequality in (\ref{eq:sass-V0}) with $\underline{\psi}=0$. Let $i\in\Zp$, $\gamma\in[\gamma_\star,1]$, $x\in\R^{n_x}$ and $\varphi_{\gamma,k}$, $u_k$ be as in the proof of Theorem \ref{th:near-optimality-conditionsV0-Vstar}. By (\ref{eq:vi-as-finite-horizon}), $V_{\gamma,i}(x)\leq\sum_{k=0}^{i-1}\gamma^k\ell(\varphi_{\gamma,k},u_k)+\gamma^iV_{\gamma,0}(\varphi_{\gamma,i})$, which, combined with (\ref{eq:proof-near-opt-dp}) and Proposition \ref{prop:conditions-sa-V-0}(ii), gives
\begin{eqn}
V_{\gamma,i+1}(x)-V_{\gamma,i}(x) & \geq & \gamma^i\big(V_{\gamma,1}(\varphi_{\gamma,i})-V_{\gamma,0}(\varphi_{\gamma,i})\big) \\
& \hspace{-2.7cm}= & \hspace{-2.7cm}\gamma^i\!\!\!\!\!\min_{v\in\mathcal{U}(\varphi_{\gamma,i})}\big(\ell(\varphi_{\gamma,i},v)+\gamma V_{\gamma,0}(f(\varphi_{\gamma,i},v))-V_{\gamma,0}(\varphi_{\gamma,i})\big)\\
& \hspace{-2.7cm}\geq & \hspace{-2.7cm}\gamma^i\alpha_0(\sigma(\varphi_{\gamma,i})).
\end{eqn}
Hence $V_{\gamma,i+1}(x)-V_{\gamma,i}(x)=|V_{\gamma,i+1}(x)-V_{\gamma,i}(x)|$ and $\sigma(\varphi_{\gamma,i})\leq\alpha_0^{-1}(\nicefrac{1}{\gamma^i}|V_{\gamma,i+1}(x)-V_{\gamma,i}(x)|)$. As $V_{\gamma,1}\geq 0$ and $V_{\gamma,\star}\leq\overline\alpha\circ\sigma$ by SA\ref{sass:stabilizability},
\begin{eqn}
V_{\gamma,\star}(\varphi_{\gamma,i})-V_{\gamma,1}(\varphi_{\gamma,i}) & \leq & \overline\alpha(\sigma(\varphi_{\gamma,i})) \\
& \leq & \overline\alpha\circ\alpha_0^{-1}(\nicefrac{1}{\gamma^i}|V_{\gamma,i+1}(x)\!-\!V_{\gamma,i}(x)|).
\end{eqn}
The first inequality in (\ref{eq:sass-V0}) holds with $\overline{\psi}(s,\gamma,i)=\overline\alpha\circ\alpha_0^{-1}(s)$. \hfill $\Box$

\noindent\textbf{Proof of Theorem \ref{th:near-optimality-bounds-summable}.} Let $\theta\in\Theta$, $\gamma\in[\gamma_\star,1]$,  $i\in\Zp$ as in (\ref{eq:stopping-criterion}) and $x\in\R^{n_x}$. By definitions of $V_{\gamma,\star}$ and $V_{\gamma,i+1}$ in (\ref{eq:optimal-value-function}) and (\ref{eq:vi-value}), respectively,
$ V_{\gamma,i+1}(x) - V_{\gamma,\star}(x) = \min_{u \in \mathcal{U}(x)} \big( \ell(x,u) + \gamma V_{\gamma,i}(f(x,u)) \big)  - \min_{u \in \mathcal{U}(x)} \big( \ell(x,u) + \gamma V_{\gamma,\star}(f(x,u)) \big) = \min_{u \in \mathcal{U}(x)}\!\! \big( \ell(x,u) \!+\! \gamma V_{\gamma,i+1}(f(x,u)) \!+\! \gamma \Lambda_i(f(x,u)) \big)  - \min_{u \in \mathcal{U}(x)} \big( \ell(x,u) + \gamma V_{\gamma,\star}(f(x,u)) \big)$
where $\Lambda_i(z):=V_{\gamma,i}(z)-V_{\gamma,i+1}(z)$ for any $z\in\R^{n_x}$. Let $h_{\gamma,\star}\in H_{\gamma,\star}$ with $H_{\gamma,\star}$ as in (\ref{eq:optimal-policy}), $V_{\gamma,\star}(x)=\ell(x,h_{\gamma,\star}(x))+\gamma V_{\gamma,\star}(f(x,h_{\gamma,\star}(x)))$, therefore
$V_{\gamma,i+1}(x) - V_{\gamma,\star}(x) \leq \ell(x,h_{\gamma,\star}(x)) + \gamma V_{\gamma,i+1}(f(x,h_{\gamma,\star}(x))) + \gamma \Lambda_i(f(x,h_{\gamma,\star}(x))) - \ell(x,h_{\gamma,\star}(x)) 
- \gamma V_{\gamma,\star}(f(x,h_{\gamma,\star}(x))) \leq \gamma \bigl( V_{\gamma,i+1}(f(x,h_{\gamma,\star}(x))) - V_{\gamma,\star}(f(x,h_{\gamma,\star}(x))) \bigr)  + \gamma \Lambda_i(f(x,h_{\gamma,\star}(x)))$.
Let $\phi_{\gamma,\star}$ be the solution to (\ref{eq:sys-optimal}) obtained with the optimal policy $h_{\gamma,\star}$, we derive by induction that for any $k\in\Zo$, 
\begin{equation}
\begin{aligned}
    V_{\gamma,i+1}(x) &- V_{\gamma,\star}(x) \leq \gamma^k \big( V_{\gamma,i+1}(\phi_{\gamma,\star}(k;x)) \\
    &  \hspace{-0.8cm}\textstyle- V_{\gamma,\star}(\phi_{\gamma,\star}(k;x)) \big) + \sum_{j=1}^{k} \gamma^j \Lambda_i(\phi_{\gamma,\star}(j;x)).
\end{aligned}
\end{equation}
As (\ref{eq:stopping-criterion}) holds, $\Lambda_i(\phi_{\gamma,\star}(j;x))\leq c_\stopvi(\phi_{\gamma,\star}(j;x),\theta,\gamma,i)$ for any $j\in\Zp$ and thus
\begin{equation} \label{eq:convergence-bound}
\begin{aligned}
    V_{\gamma,i+1}(x) &- V_{\gamma,\star}(x) \leq \gamma^k \bigl( V_{\gamma,i+1}(\phi_{\gamma,\star}(k;x)) \\
    &\textstyle- V_{\gamma,\star}(\phi_{\gamma,\star}(k;x)) \bigr) + \sum\limits_{j=1}^{k} \gamma^j c_{\stopvi}(\phi_{\gamma,\star}(j;x),\theta, \gamma, i).
\end{aligned}
\end{equation}
As $V_{\gamma,\star}(\phi_{\gamma,\star}(k;x))\geq 0$ and $V_{\gamma,i+1}(\phi_{\gamma,\star}(k;x))\leq \overline\alpha(\sigma(\phi_{\gamma,\star}(k;x)))$ by SA\ref{sass:stabilizability},
\begin{equation} \label{eq:v-convergence-final}
\begin{aligned}
    V_{\gamma,i+1}(x) &- V_{\gamma,\star}(x) \leq \gamma^k \overline{\alpha} \bigl( \sigma(\phi_{\gamma,\star}(k;x)) \bigr) \\
    & \textstyle+ \sum_{j=1}^{k} \gamma^j c_{\stopvi}(\phi_{\gamma,\star}(j;x),\theta, \gamma, i).
\end{aligned}
\end{equation}
By Lemma \ref{lem:bound-optimal-solutions}, we derive
$   V_{\gamma,i+1}(x) - V_{\gamma,\star}(x) \leq \gamma^k \overline{\alpha} \big( \xi(\sigma(x), k)  \big) + \sum_{j=1}^{k} \gamma^j c_{\stopvi}(\phi_{\gamma,\star}(j;x), \theta,\gamma, i)$. 
Taking the limit as $k\to\infty$, we have as $\xi$ is non-increasing and converging to zero in its second argument and $\gamma^k \leq 1$,
\begin{equation} \label{eq:proof-near-optimality-summable-one-side}
\begin{aligned}
    \textstyle V_{\gamma,i+1}(x) - V_{\gamma,\star}(x) \leq \lim_{k\to\infty} \sum\limits_{j=1}^{k} \gamma^j c_\stopvi(\phi_{\gamma,\star}(j;x), \theta,\gamma, i).
\end{aligned}
\end{equation}
On the other hand,  
$V_{\gamma,\star}(x) - V_{\gamma,i+1}(x) = \min_{u\in\mathcal{U}(x)} \big( \ell(x,u) + \gamma V_{\gamma,\star}(f(x,u)) \big) - \min_{u\in\mathcal{U}(x)} \big( \ell(x,u) + \gamma V_{\gamma,i}(f(x,u)) \big)$. Let $h_{\gamma,i+1}\in H_{\gamma,i+1}$ with $H_{\gamma,i+1}$ in (\ref{eq:vi-policy}). As $V_{\gamma,i+1}(x)=\ell(x,h_{\gamma,i+1}(x))+\gamma V_{\gamma,i+1}(f(x,h_{\gamma,i+1}(x)))+\gamma \Lambda_i(f(x,h_{\gamma,i+1}(x)))$, 
\begin{equation}
\begin{aligned}
    V_{\gamma,\star}(x) &- V_{\gamma,i+1}(x) \\
    &\leq \ell(x,h_{\gamma,i+1}(x)) + \gamma V_{\gamma,\star}(f(x,h_{\gamma,i+1}(x))) \\
    &\quad - \ell(x,h_{\gamma,i+1}(x)) - \gamma V_{\gamma,i+1}(f(x,h_{\gamma,i+1}(x))) \\
    &\quad - \gamma \Lambda_i(f(x,h_{\gamma,i+1}(x))) \\
    &\leq \gamma \bigl( V_{\gamma,\star}(f(x,h_{\gamma,i+1}(x))) \!-\! V_{\gamma,i+1}(f(x,h_{\gamma,i+1}(x))) \bigr) \\
    &\quad - \gamma \Lambda_i(f(x,h_{\gamma,i+1}(x))).
\end{aligned}
\end{equation}
Let $\phi_{\gamma,i+1}$ be the solution to (\ref{eq:sys-vi}) with $i+1$ obtained by applying policy $h_{\gamma,i+1}$, by induction,  for any $k\in\Zo$, 
\begin{equation}
\begin{aligned}
    V_{\gamma,\star}(x) &- V_{\gamma,i+1}(x) \\
    &\leq \gamma^k \big( V_{\gamma,\star}(\phi_{\gamma,i+1}(k;x)) - V_{\gamma,i+1}(\phi_{\gamma,i+1}(k;x)) \big) \\
    &\quad \textstyle- \sum_{j={1}}^k \gamma^j \Lambda_i(\phi_{\gamma,i+1}(j;x)) \\
    &\leq \textstyle \gamma^k V_{\gamma,\star}(\phi_{\gamma,i+1}(k;x)) - \sum_{j={1}}^k \gamma^j \Lambda_i(\phi_{\gamma,i+1}(j;x)).
\end{aligned}
\end{equation}
We have that $V_{\gamma,\star}(\phi_{\gamma,i+1}(k;x))\leq\overline\alpha(\sigma(\phi_{\gamma,i+1}(k;x)))$ and $\sigma(\phi_{\gamma,i+1}(k;x))\to 0$ as $k\to\infty$ by {assumption}. Moreover, (\ref{eq:stopping-criterion}) holds {at any $z\in \R^{n_x}$}, hence $-\Lambda_i(\phi_{\gamma,i+1}(j;x))\leq c_\stopvi(\phi_{\gamma,i+1}(j;x),\theta,\gamma,i)$. Consequently, 
\begin{equation} \label{eq:proof-near-optimality-summable-other-side}
    V_{\gamma,\star}(x) - V_{\gamma,i+1}(x) \leq \textstyle\lim_{k \to \infty} \sum\limits_{j={1}}^{k}\!\! \gamma^j c_\stopvi(\phi_{\gamma,i+1}(j;x),\theta, \gamma, i).
\end{equation}
The desired result follows by combining (\ref{eq:proof-near-optimality-summable-one-side}) and (\ref{eq:proof-near-optimality-summable-other-side}). \hfill $\Box$

\subsection{Technical lemmas}

\begin{lem}\label{lem:bound-optimal-solutions} For all  $\gamma\in[\gamma_\star,1]$ with $\gamma_\star$ from SA\ref{sass:optimal-policy-ugas}, all $x\in\R^{n_x}$, any solution $\phi_{\gamma,\star}$ to (\ref{eq:sys-optimal}) verifies $\sigma(\phi_{\gamma,\star}(k;x))\leq \xi(\sigma(x),k)$ for any $k\in\Zo$, where  $\xi:(s,k)\mapsto\underline\alpha^{-1}\circ(\id-\alpha\circ\overline\alpha^{-1})^{(k)}\circ\overline\alpha(s)$ is continuous in $s$ and {non-increasing and converging to zero in $k$}. 
\end{lem}

\begin{proof} Let $\gamma\in[\gamma_\star,1]$, $x\in\R^{n_x}$  and $h_{\gamma,\star}(x)\in H_{\gamma,\star}(x)$ with $H_{\gamma,\star}$ in (\ref{eq:optimal-policy}). Let $Y_{\gamma,\star}(x)=V_{\gamma,\star}(f(x,h_{\gamma,\star}(x))) - V_{\gamma,\star}(x)$. By Bellman equation, 
\begin{eqn}
V_{\gamma,\star}(x) & = & \ell(x,h_{\gamma,\star}(x)) + \gamma V_{\gamma,\star}(f(x,h_{\gamma,\star}(x))). 
\end{eqn}
Consequently, by SA\ref{sass:detectability} and SA\ref{sass:stabilizability},
\begin{eqn}
Y_{\gamma,\star}(x) & = & -\frac{1}{\gamma}\ell(x,h_{\gamma,\star}(x)) + \frac{1-\gamma}{\gamma}V_{\gamma,\star}(x) \\
& \leq & -\frac{1}{\gamma}\underline\alpha(\sigma(x))+\frac{1-\gamma}{\gamma}\overline\alpha(\sigma(x)).
\end{eqn}
By SA\ref{sass:optimal-policy-ugas}, {$-\frac{1}{\gamma}\underline\alpha(\sigma(x)) \leq -\frac{1}{\gamma}\alpha(\sigma(x))- \frac{1-\gamma_\star}{\gamma}\overline\alpha(\sigma(x))$, hence $Y_{\gamma,\star}(x) \leq   -\frac{1}{\gamma}\alpha(\sigma(x))- \frac{1-\gamma_\star}{\gamma}\overline\alpha(\sigma(x))+\frac{1-\gamma}{\gamma}\overline\alpha(\sigma(x))\leq -\frac{1}{\gamma}\alpha(\sigma(x))$. Consequently, as $\gamma\in[ \gamma_\star,1]$ } 
and by SA\ref{sass:stabilizability},
\begin{eqn}\label{eq:proof-lemma-bound-Y-optimal}
Y_{\gamma,\star}(x) 
& \leq & -\alpha\circ \overline\alpha^{-1}(V_{\gamma,\star}(x)).
\end{eqn}
Hence 
$V_{\gamma,\star}(f(x,h_{\gamma,\star}(x)))  \leq  (\id-\alpha\circ\overline\alpha^{-1})(V_{\gamma,\star}(x))$. 
Let $\phi_{\gamma,\star}$ be any solution to (\ref{eq:sys-optimal}) initialized at $x$ at time $0$. {We consider that $\id-\alpha\circ\overline\alpha^{-1}$ is non-decreasing without loss of generality \cite[Remark 2]{Granzotto-et-al-tac-finite-discounted-horizon}. Hence} by iteratively applying the last inequality, we derive that
\begin{eqn}
V_{\gamma,\star}(\phi_{\gamma,\star}(k;x))  & \leq & (\id-\alpha\circ\overline\alpha^{-1})^{(k)}(V_{\gamma,\star}(x)).
\end{eqn}
We have that  $V_{\gamma,\star}(z)\geq \ell(z,u)$ for any $z\in\R^{n_x}$ and $u\in H_{\gamma,\star}(z)$, hence $V_{\gamma,\star}(z)\geq \ell(z,u)\geq \underline\alpha(\sigma(z))$ by SA\ref{sass:detectability}. As a result, recalling that $V_{\gamma,\star}(x)\leq \overline\alpha(\sigma(x))$ by SA\ref{sass:stabilizability},
$\sigma(\phi_{\gamma,\star}(k;x))  \leq \underline\alpha^{-1}\circ(\id-\alpha\circ\overline\alpha^{-1})^{(k)}\circ\overline\alpha(\sigma(x)) 
 = \xi(\sigma(x),k)$. 
We are left with proving that $\xi$ is continuous in $s$, non-increasing and converging to  $0$ in $k$. As $\underline\alpha,\overline\alpha,\alpha\in\Kinf$, this is equivalent to proving that $\varsigma:(s,k)\mapsto(\id-\alpha\circ\overline\alpha^{-1})^{(k)}(s)$ is continuous in $s$ and non-increasing and converging to $0$ in $k$. We have that $\varsigma(\cdot,k)$ is continuous for any $k\in\Zo$ being the composition of continuous functions. On the other hand, let $s\in\Rlp$, $s-\alpha\circ\overline\alpha^{-1}(s)\leq s$ as $\alpha\circ\overline\alpha^{-1}(s)\geq 0$. We derive by iteration that $(\id-\alpha\circ\overline\alpha^{-1})^{(k+1)}(s)\leq (\id-\alpha\circ\overline\alpha^{-1})^{(k)}(s)$ for any $s\geq 0$, which means that $\varsigma(s,\cdot)$ is non-increasing. Consequently, as $\varsigma$ takes non-negative values, given any $s>0$, there exists $c\geq 0$ such that $\varsigma(s,k)\to c$ as $k\to\infty$. Furthermore, $c$ is a fixed point of $\id-\alpha\circ\overline\alpha^{-1}$, namely  $c=c-\alpha\circ\overline\alpha^{-1}(c)$, which is equivalent to $\alpha\circ\overline\alpha^{-1}(c)=0$, which is equivalent to $c=0$ as $\alpha\circ\overline\alpha^{-1}\in\Kinf$. We have proved that $\varsigma(s,\cdot)$ {is non-increasing and converging  to $0$ at infinity} for any $s\geq 0$. As a consequence, so does $\xi(s,\cdot)$ for any $s\geq 0$, which completes the proof.
\end{proof}

We derive the next lemma inspired by \cite[Thm. 5]{Granzotto-et-al-tac-finite-discounted-horizon}.

\begin{lem}\label{lem:bound-vi-solutions} For all  $\gamma\in[\gamma_\star,1]$ with $\gamma_\star$ from SA\ref{sass:optimal-policy-ugas}, all $x\in\R^{n_x}$, all $i\in\Zp$, and all sequences  $\varphi_{\gamma,k}$, $k\in\{0,\ldots,i\}$, verifying $\varphi_{\gamma,0}=x$ and  $\varphi_{\gamma,k+1}\in F_{\gamma,i-k}(\varphi_{\gamma,k})$ for $k\in\{0,\ldots,i-1\}$ with $F_{\gamma,i-k}$ as in (\ref{eq:sys-vi}), it holds that $\sigma(\varphi_{\gamma,k})\leq \xi(\sigma(x),k)$ for any $k\in\{0,\ldots,i{-1}\}$ with $\xi$ as in Lemma \ref{lem:bound-optimal-solutions}.
\end{lem}

\begin{proof} Let $\gamma\in[\gamma_\star,1]$, $x\in\R^{n_x}$, $i\in\Zp$, $k\in\{0,\ldots,{i-2}\}$, a sequence $\varphi_{\gamma,j}$, $j\in\{0,\ldots,i\}$ as in Lemma \ref{lem:bound-vi-solutions} {and $h_{\gamma,i-k}\in H_{\gamma,i-k}$ the corresponding selections used to generate this sequence}. By (\ref{eq:vi-value}), we have by definition of the sequence $\varphi_{\gamma,k}$,
\begin{eqn}
V_{\gamma,i-k}(\varphi_{\gamma,k}) & = & \ell(\varphi_{\gamma,k},h_{\gamma,i-k}(\varphi_{\gamma,k}))\\
& & +\gamma V_{\gamma,i-(k+1)}(f(\varphi_{\gamma,k},h_{\gamma,i-k}(\varphi_{\gamma,k})))\\
& = & \ell(\varphi_{\gamma,k},h_{\gamma,i-k}(\varphi_{\gamma,k}))\!+\!\gamma V_{\gamma,i-(k+1)}(\varphi_{\gamma,k+1}).
\end{eqn}
Hence 
$V_{\gamma,i-(k+1)}(\varphi_{\gamma,k+1}) - V_{\gamma,i-k}(\varphi_{\gamma,k}) =  - \frac{1}{\gamma} \ell(\varphi_{\gamma,k},h_{\gamma,i-k}(\varphi_{\gamma,k})) 
 +\frac{1-\gamma}{\gamma} V_{\gamma,i-k}(\varphi_{\gamma,k})$ 
and, like in the proof of Lemma \ref{lem:bound-optimal-solutions}, by SA\ref{sass:detectability}, SA\ref{sass:stabilizability} and SA\ref{sass:optimal-policy-ugas},
$
V_{\gamma,i-(k+1)}(\varphi_{\gamma,k+1}) - V_{\gamma,i-k}(\varphi_{\gamma,k})  \leq - \alpha\circ\overline\alpha^{-1}(V_{\gamma,i-k}(\varphi_{\gamma,k}))$ 
equivalently,
$V_{\gamma,i-(k+1)}(\varphi_{\gamma,k+1}) \leq (\id-  \alpha\circ\overline\alpha^{-1})(V_{\gamma,i-k}(\varphi_{\gamma,k}))$. 
{The case $k=0$ of the statement holds as $\varphi_{\gamma,0}=x$ and $\xi(s,0)\geq s$ for any $s\geq0$.} The proof then follows that of Lemma \ref{lem:bound-optimal-solutions}.
\end{proof}

Using Lemmas \ref{lem:bound-optimal-solutions} and \ref{lem:bound-vi-solutions}, we derive bounds on $V_{\gamma,\star}$ and $V_{\gamma,i}$.

\begin{lem}\label{lem:bound-V-star-Vi} For all $\gamma\in[\gamma_\star,1]$ with $\gamma_\star$ from SA\ref{sass:optimal-policy-ugas}, all  {$i\in\Zp$} and all $x\in\R^{n_x}$,
$V_{\gamma,i}(x)  \leq   V_{\gamma,\star}(x) + \overline\alpha\left(\xi(\sigma(x),{i-1})\right)$ and 
$V_{\gamma,\star}(x)  \leq   V_{\gamma,i}(x) + \overline\alpha\left(\xi(\sigma(x),{i-1})\right)$,
with $\xi$ as in Lemma \ref{lem:bound-optimal-solutions}. 
\end{lem}

\begin{proof} Let $\gamma\in[\gamma_\star,1]$, $x\in\R^{n_x}$, ${i\in\Zp}$ and $h_{\gamma,\star}\in H_{\gamma,\star}$ with $H_{\gamma,\star}$ in (\ref{eq:optimal-policy}). We denote by $\phi_{\gamma,\star}$ the solution to (\ref{eq:sys-optimal}) obtained with the optimal policy $h_{\gamma,\star}$. By (\ref{eq:proof-vi-as-mpc}),
\begin{eqn}
V_{\gamma,i}(x) & \leq & \sum_{k=0}^{i-1}\gamma^k \ell(\phi_{\gamma,\star}(k;x),h_{\gamma,\star}(\phi_{\gamma,\star}(k;x)))\\
& & +\gamma^i V_{\gamma,0}(\phi_{\gamma,\star}(i;x)).
\end{eqn}
By SA\ref{sass:stabilizability}, $V_{\gamma,0}(\phi_{\gamma,\star}(i;x))\leq\overline\alpha(\sigma(\phi_{\gamma,\star}(i;x)))$ and, as $\sum_{k=0}^{i-1}\gamma^k \ell(\phi_{\gamma,\star}(k;x),h_{\gamma,\star}(\phi_{\gamma,\star}(k;x)))\leq V_{\gamma,\star}(x)$ by (\ref{eq:optimal-value-function}), 
$
V_{\gamma,i}(x)  \leq  V_{\gamma,\star}(x)+\overline\alpha(\sigma(\phi_{\gamma,\star}(i;x)))$.  
By invoking Lemma \ref{lem:bound-optimal-solutions} {and as $\xi$ is non-increasing in its second argument,}  we derive 
$V_{\gamma,i}(x) \leq  V_{\gamma,\star}(x) + \overline\alpha\left(\xi(\sigma(x),i)\right)
 \leq   V_{\gamma,\star}(x) + \overline\alpha\left(\xi(\sigma(x),i-1)\right)$.
Now, let $\mathbf{u}=(u_0,\ldots,u_{i-1})\in(\R^{n_u})^i$ be such that 
\begin{eqn}\label{eq:proof-lem-Vi-mpc}
V_{\gamma,i}(x) & = & \min_{\mathbf{u}}\big(\sum_{k=0}^{i-1}\gamma^k \ell(\phi(k;x,\mathbf{u}\vert_k),u_k)  \\  & & \hspace{1cm}+\gamma^{i} V_{\gamma,0}(\phi(i;x,\mathbf{u}\vert_{i}))\big)
\end{eqn}
as in (\ref{eq:vi-as-finite-horizon}). Given the definitions of $V_{\gamma,\star}$ in (\ref{eq:optimal-value-function}){, of $V_{\gamma,i}$ in (\ref{eq:vi-value}) and by SA\ref{sass:stabilizability},}
\begin{eqn}
V_{\gamma,\star}(x) & \leq & \sum_{k=0}^{i-{2}}\gamma^k\ell(\phi(k;x,\mathbf{u}\vert_k),u_k) \\
& & + \gamma^{{i-1}} V_{\gamma,\star}(\phi(i{-1};x,\mathbf{u}{\vert_{i-1}}))\\
& {\leq} & {V_{\gamma,i}(x)+\gamma^{i-1}\overline\alpha(\sigma(\phi(i{-1};x,\mathbf{u}{\vert_{i-1}})))}\\
& {\leq} & {V_{\gamma,i}(x)+\overline\alpha(\sigma(\phi(i{-1};x,\mathbf{u}{\vert_{i-1}})))}
\end{eqn}
We note that $\phi(k+1;x,\mathbf u|_{k+1})\in F_{\gamma,i-k}(\phi(k;x,\mathbf u|_k))$ for $k\in\{0,\dots,i-1\}$ {by the principle of optimality $u_k\in H_{\gamma,i-k}(\phi(k;x,\mathbf{u}\vert_k))$, the tail of optimal sequence being optimal for the remaining horizon}. This means $\phi(\cdot;x,\mathbf{u}{\vert}_{\cdot})$, $k\in\{0,\ldots,i-1\}$, corresponds to a sequence $\varphi_{\gamma,k}$, $k\in\{0,\ldots,i-1\}$ as defined in Lemma \ref{lem:bound-vi-solutions}. Consequently, by Lemma \ref{lem:bound-vi-solutions}, 
$V_{\gamma,\star}(x) \leq 
 V_{\gamma,i}(x)+ \overline\alpha\left(\xi(\sigma(x),i{-1})\right)$.%
\end{proof}

\section*{{Acknowledgment}}
{Claude Opus 5.5 was used during the revision to check consistency, detect typos and notational errors, and improve the wording. All suggestions were verified by the authors, who take full responsibility for the content.}

\section*{References}

\bibliographystyle{ieeetr}
\bibliography{newbib.bib}

\begin{IEEEbiography}[{\includegraphics[width=1in,height=1.25in,clip,keepaspectratio]{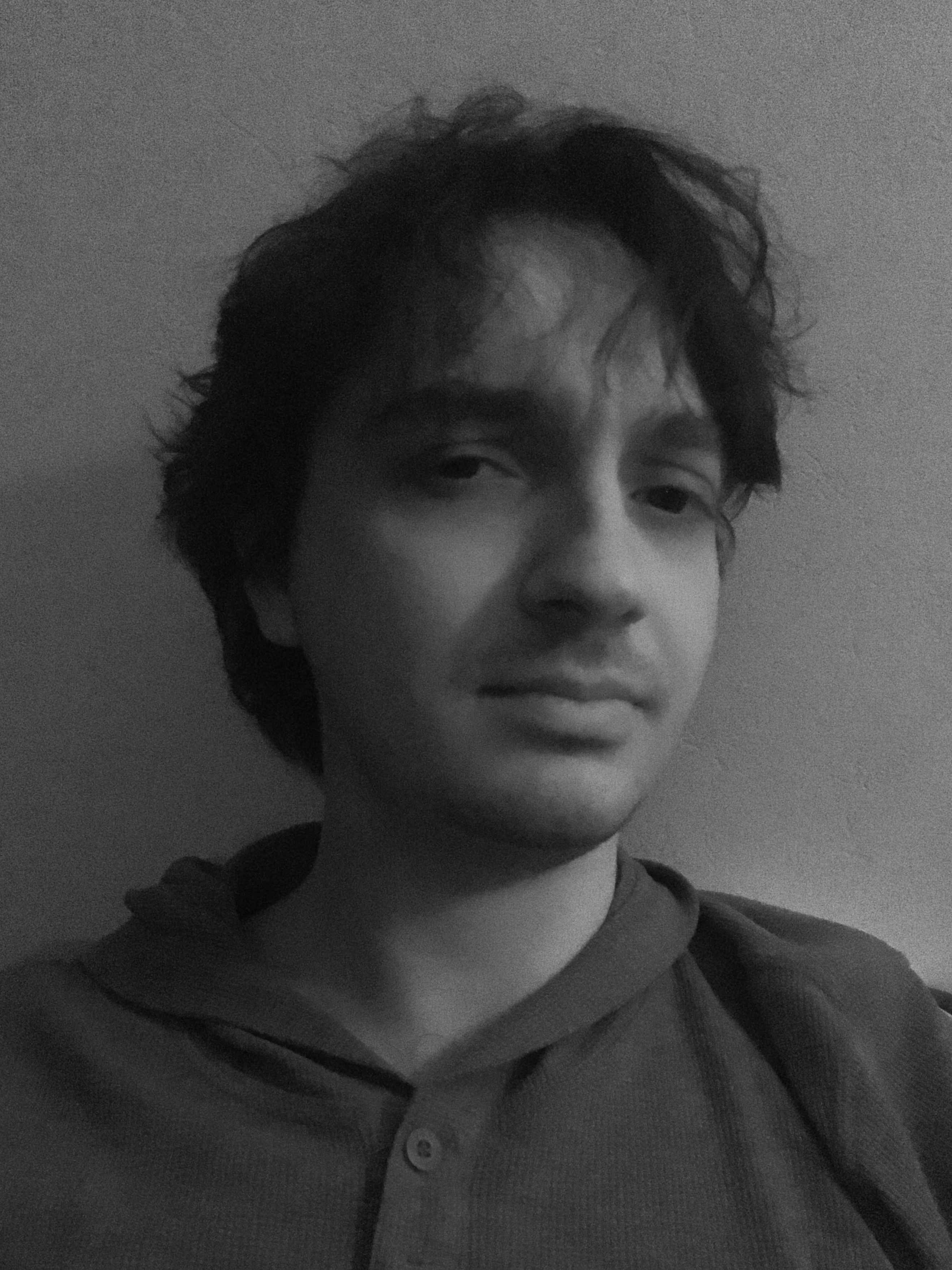}}]{Mathieu Granzotto} (Member, IEEE) received his engineering degree in Control and Automation in 2016 from “Universidade Federal de Santa Catarina”, Brazil. In 2019, he received his Ph.D. in Control Theory from Université de Lorraine, France, where he was a Temporary Research
and Teaching Attaché at CRAN. Since 2022, he has been a research fellow at the Department of Electrical and Electronic Engineering (DEEE) at the University of Melbourne, Australia. 
\end{IEEEbiography}

\begin{IEEEbiography}[{\includegraphics[width=1in,height=1.25in,clip,keepaspectratio]{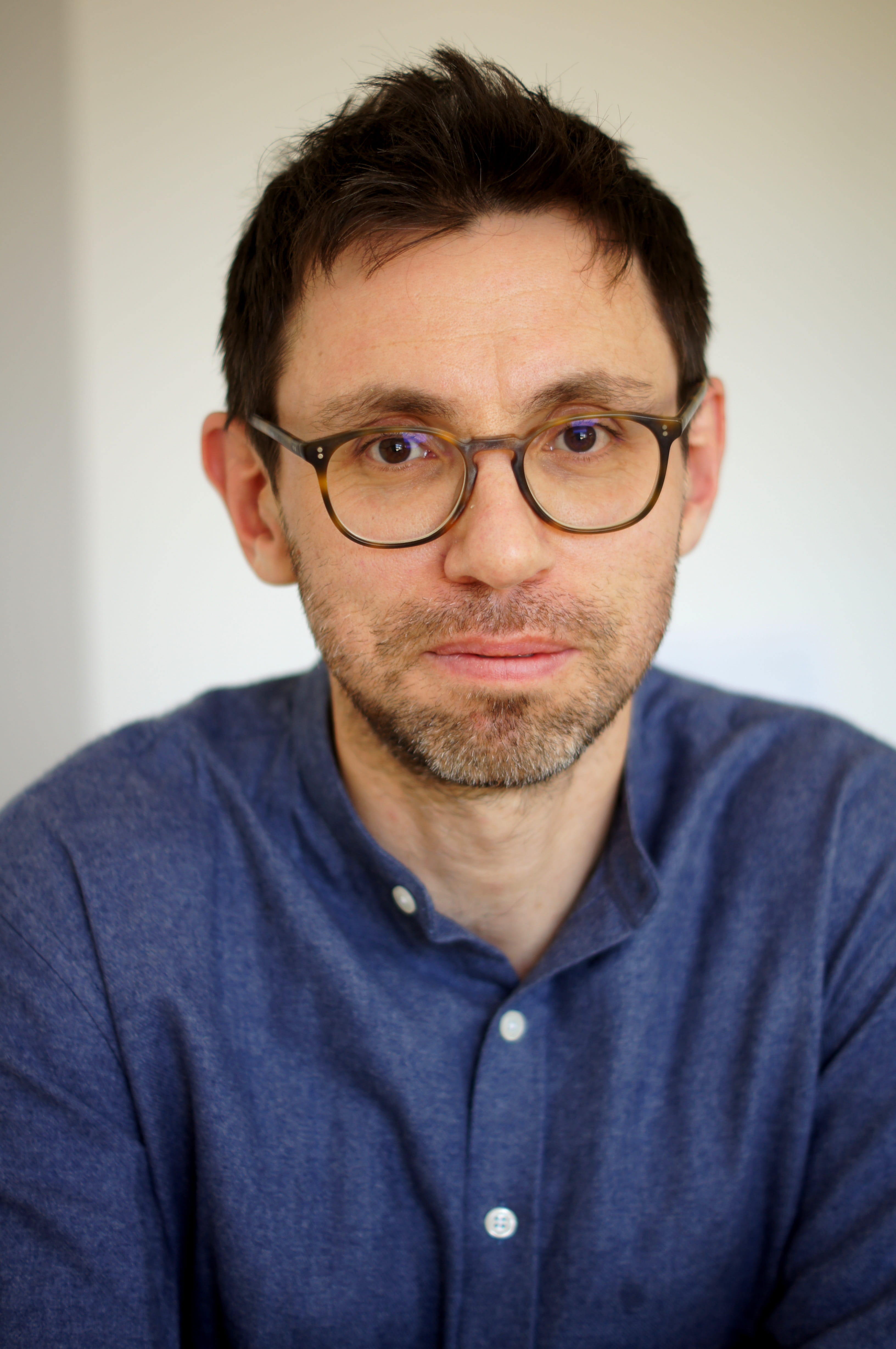}}]{Romain Postoyan} (Senior Member, IEEE) received the ``Ing\'enieur'' degree in Electrical and Control Engineering from ENSEEIHT (France) in 2005. He obtained the M.Sc. by Research in Control Theory \& Application from Coventry University (United Kingdom) in 2006 and the Ph.D. in Control Engineering from Universit\'e Paris-Sud (France) in 2009. In 2010, he was a research assistant at the University of Melbourne (Australia). Since 2011, he has been a CNRS researcher at CRAN (France). 
\end{IEEEbiography}

\begin{IEEEbiography}[{\includegraphics[width=1in,height=1.25in,clip,keepaspectratio]{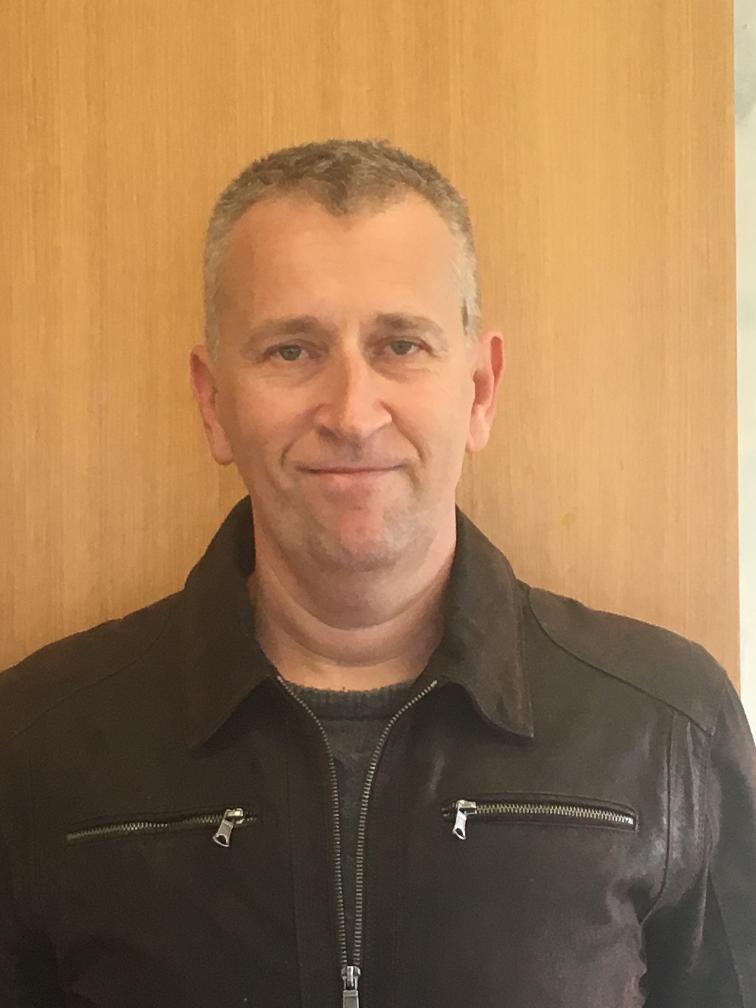}}]
{Dragan Neši\'c} (Fellow, IEEE)
is a Professor at  The University of Melbourne, Australia. He currently serves as Associate Dean Research at the Melbourne School of Engineering.  His research interests include networked control systems, reset systems, extremum seeking control, hybrid control systems, event-triggered control, security and privacy in cyber-physical systems, and so on. 
He is a Fellow of the Institute of Electrical and Electronic Engineers (IEEE, 2006) and Fellow of the International Federation for Automatic Control (IFAC, 2019). 
\vspace{-1cm}
\end{IEEEbiography}
\begin{IEEEbiography}[{\includegraphics[width=1in,height=1.25in,clip,keepaspectratio]{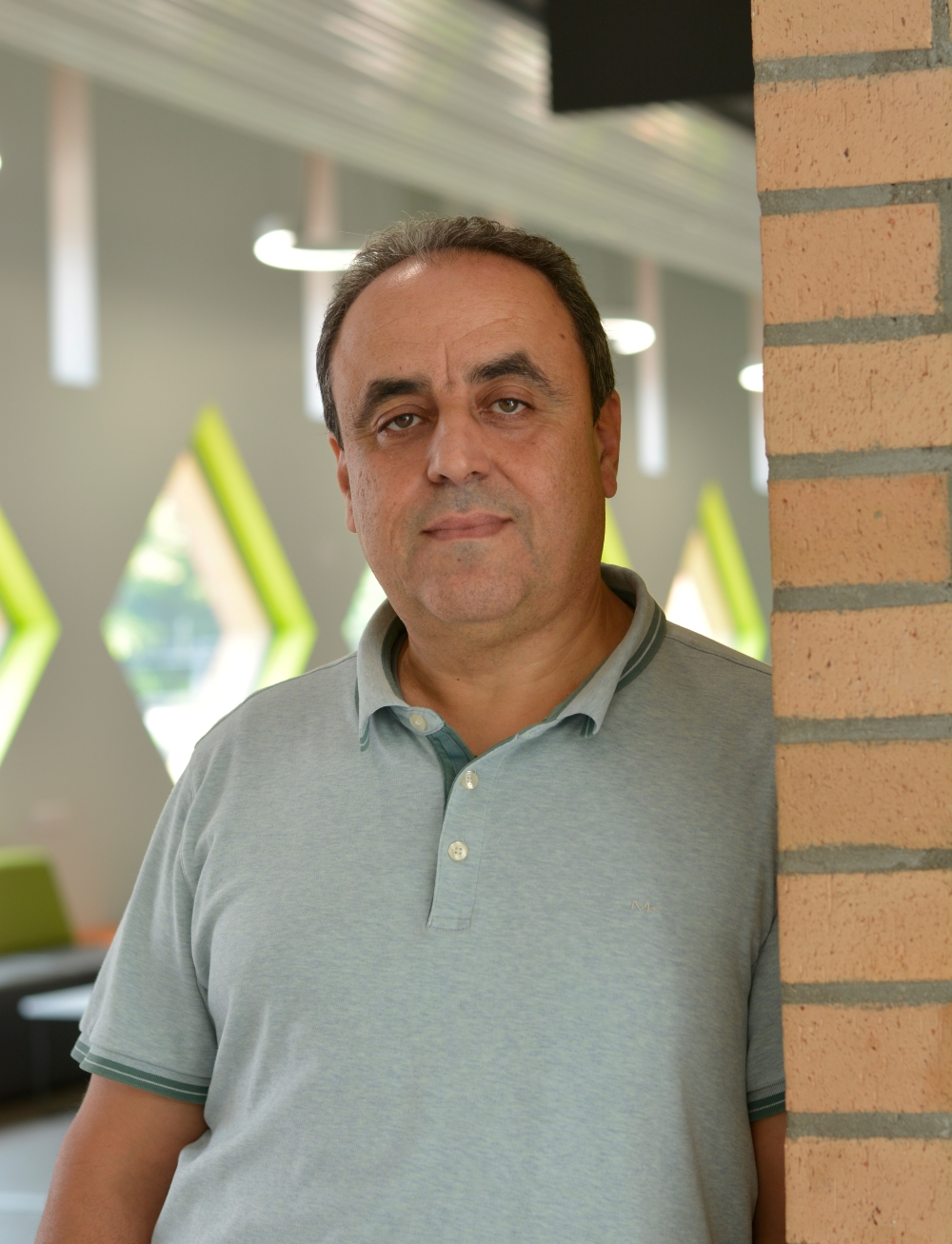}}]{Jamal Daafouz}
is a Full Professor at the Université de Lorraine, a member of the Institut Universitaire de France (IUF), and a Researcher at CRAN CNRS. His research interests include analysis, observation, and control of switched systems and networked systems with a particular interest in convex-based optimization methods.
Dr. Daafouz served as a Senior Editor of IEEE Control Systems Letters and Associate Editor for Automatica and IEEE Transactions on Automatic Control.\vspace{-1cm}
\end{IEEEbiography}

\begin{IEEEbiography}[{\includegraphics[width=1in,height=1.25in,clip,keepaspectratio]{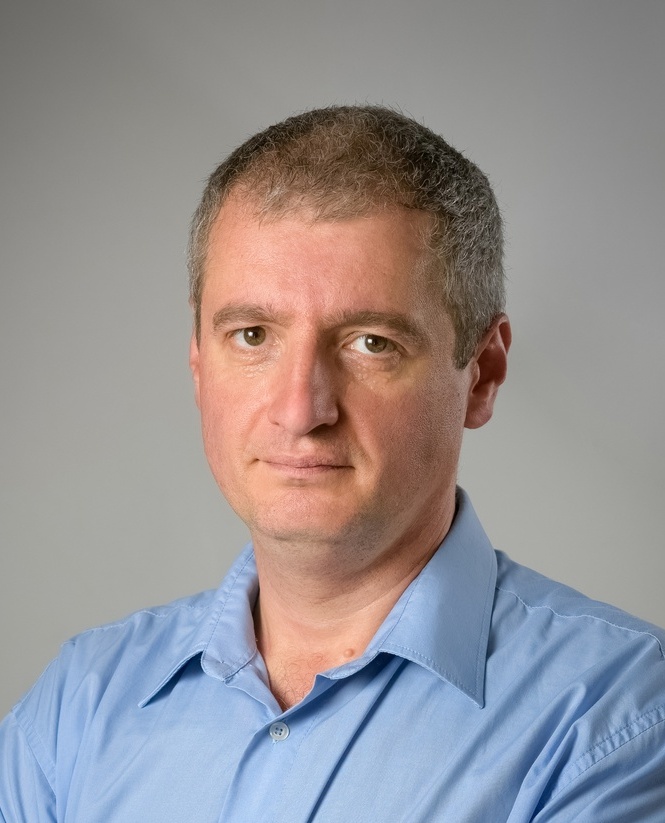}}]{Lucian Bu\c soniu}
received his Ph.D. degree cum laude from the Delft University of Technology, the Netherlands, in 2009. He is a professor with the Technical University of Cluj-Napoca, where he leads the group on Robotics and Nonlinear Control, and a Corresponding Member of the Romanian Academy. His research interests include nonlinear optimal control using artificial intelligence and reinforcement learning techniques, robotics, and multiagent systems. He serves on the editorial board of the IEEE Control Systems Letters.
\end{IEEEbiography}

\end{document}